\documentclass[12pt]{amsart}

\usepackage[utf8]{inputenc}

\usepackage{amssymb,amsthm,amsmath,amsxtra,mathrsfs}
\usepackage[all]{xy}
\usepackage{fullpage}
\usepackage{comment}
\usepackage{hyperref}
\usepackage{color}
\usepackage{graphicx}
\usepackage{booktabs}
\usepackage{mathtools}
\usepackage{algorithm}
\usepackage[noend]{algorithmic}
\usepackage[all]{xy}
\usepackage{enumerate}
\usepackage{multicol}

\graphicspath{ {images/} }

\numberwithin{equation}{section}
  
\theoremstyle{plain}
\newtheorem{thm}[equation]{Theorem}
\newtheorem{prop}[equation]{Proposition}
\newtheorem{lem}[equation]{Lemma} 
\newtheorem{cor}[equation]{Corollary}

\newtheorem*{cor*}{Corollary}
\newtheorem*{prob*}{Problem}
\newtheorem*{thm*}{Theorem}
\newtheorem*{thma*}{Theorem A}
\newtheorem*{thmb*}{Theorem B}
\newtheorem*{thmc*}{Theorem C}
\newtheorem*{conj*}{Conjecture}

\theoremstyle{definition}
\newtheorem{defn}[equation]{Definition}

\theoremstyle{remark}

\newtheorem{rmk}[equation]{Remark}
\newtheorem{claim}[equation]{Claim}

\newenvironment{enumalph}
{\begin{enumerate}}
{\end{enumerate}}

\newenvironment{enumroman}
{\begin{enumerate}}
{\end{enumerate}}

\DeclareMathOperator{\adim}{adim}

\DeclareMathOperator{\Gal}{Gal}

\DeclareMathOperator{\M}{M}

\DeclareMathOperator{\ord}{ord}

\DeclareMathOperator{\PSL}{PSL}

\DeclareMathOperator{\SL}{SL}

\DeclareMathOperator{\Tr}{Tr}
\DeclareMathOperator{\NDP}{NDP}

\newcommand{\HH}{\mathbb H}
\newcommand{\Q}{\mathbb Q}
\newcommand{\R}{\mathbb R}

\newcommand{\Z}{\mathbb Z}

\newcommand{\Belyi}{Bely\u{\i}}

\newcommand{\calH}{\mathcal{H}}

\newcommand{\calT}{\mathcal{T}}

\newcommand{\quat}[2]{\displaystyle{\biggl(\frac{#1}{#2}\biggr)}}

\newcommand{\la}{\langle}

\newcommand{\lcm}{\operatorname{lcm}}

\newcommand{\defi}{\textsf}

\def\phi{\varphi}
\def\rho{\varrho}

\begin{document}

\title{On the arithmetic dimension of triangle groups} 

\author{Steve Nugent}
\address{Lady Margaret Hall College, Norham Gardens, Oxford OX2 6QA, United Kingdom}
\email{steve.nugent@me.com}

\author{John Voight}
\address{Department of Mathematics, Dartmouth College, 6188 Kemeny Hall,
Hanover, NH 03755, USA}
\email{jvoight@gmail.com}
\date{\today}

\begin{abstract}
Let $\Delta=\Delta(a,b,c)$ be a hyperbolic triangle group, a Fuchsian group obtained from reflections in the sides of a triangle with angles $\pi/a,\pi/b,\pi/c$ drawn on the hyperbolic plane.  We define the arithmetic dimension of $\Delta$ to be the number of split real places of the quaternion algebra generated by $\Delta$ over its (totally real) invariant trace field.  Takeuchi has determined explicitly all triples $(a,b,c)$ with arithmetic dimension $1$, corresponding to the arithmetic triangle groups.  We show more generally that the number of triples with fixed arithmetic dimension is finite, and we present an efficient algorithm to completely enumerate the list of triples of bounded arithmetic dimension.
\end{abstract}

\maketitle

Classically, tessellations of the sphere, the Euclidean plane, and the hyperbolic plane by triangles  \cite{FrickeKlein,Magnus} were a source of significant interest.  Let $a,b,c \in \Z_{\geq 2} \cup \{\infty\}$, and let 
\[ \chi(a,b,c)=\left(\frac{1}{a}+\frac{1}{b}+\frac{1}{c}\right)-1. \]
Let $T=T(a,b,c)$ be a triangle with angles $\pi/a,\pi/b,\pi/c$ (where $\pi/\infty=0$); without loss of generality, we may assume $a \leq b \leq c$.  Then $T$ can be drawn in a geometry according to the excess $\chi(a,b,c)\pi$ of the sum of its angles: on the sphere if $\chi(a,b,c)>0$, the Euclidean plane if $\chi(a,b,c)=0$, and the hyperbolic plane if $\chi(a,b,c)<0$, and we call the triple $(a,b,c)$ \defi{spherical}, \defi{Euclidean}, or \defi{hyperbolic}, accordingly.  The triangle $T$ is unique up to similarity if it is Euclidean, and otherwise it is unique up to isometry, and $T$ then provides a tessellation.  The spherical and Euclidean triples have been understood since classical antiquity---giving rise to Platonic solids and familiar tessellations of the Euclidean plane---so we suppose now that the triple $(a,b,c)$ is hyperbolic, and $\chi(a,b,c)<0$.  For example, the tessellation for $(a,b,c)=(2,3,9)$ is as follows:
\begin{center}
\includegraphics{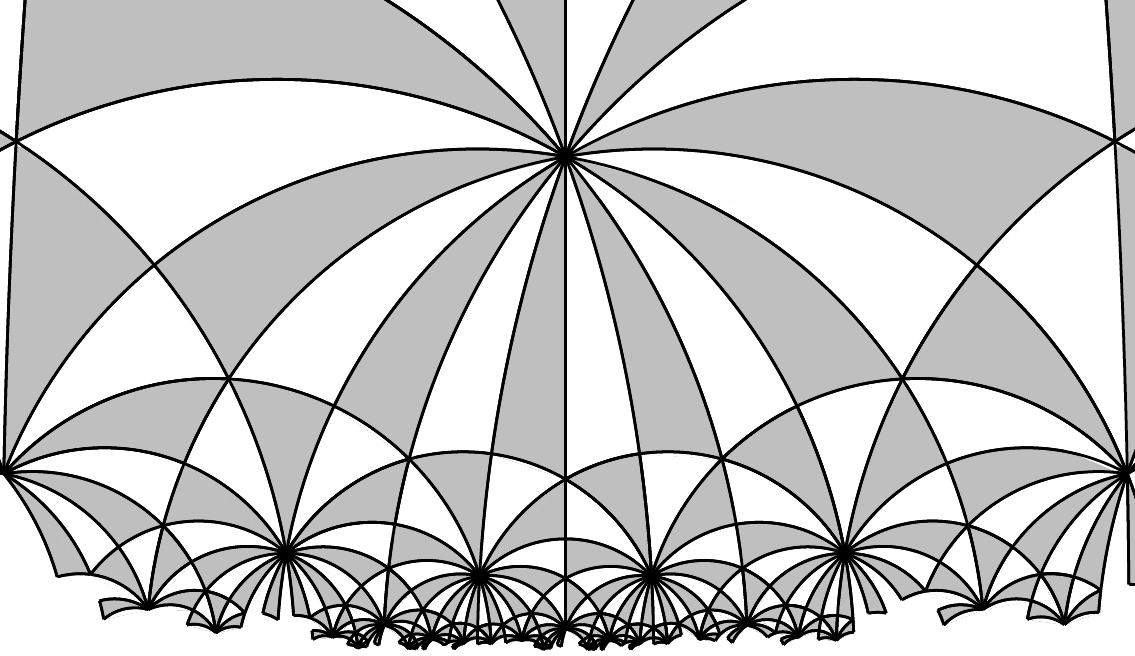}
\end{center}

In this article, we study \emph{arithmetic} properties of triangular tessellations of the hyperbolic plane, and in particular we explore the following question: what happens when we multiply the angles of the corresponding triangle by an integer factor, corresponding to the Galois action on its angles?  

Let $m=\lcm(\{a,b,c\} \smallsetminus \{\infty\})$, and for $\theta \in \R$, let $\angle(\theta) \in [0,\pi/2]$ satisfy $\cos(\angle(\theta))=\lvert\cos(\theta)\rvert$.  For $k \in (\Z/2m\Z)^\times$, we define the $k$th \defi{conjugate} triangle of $T$ to be the triangle with angles $\angle(k\pi/a),\angle(k\pi/b),\angle(k\pi/c)$.  Equivalently, we may identify an acute triangle with the cosines of its angles, and so the conjugate triangle corresponds to the conjugates of the triple of angles $(\cos \pi/a, \cos \pi/b, \cos \pi/c)$  under $\Gal(\Q(\zeta_{2m})/\Q) \simeq (\Z/2m\Z)^\times$.  Having ``swollen''  (or ``dilated'') the hyperbolic triangle $T$, the resulting conjugate triangles could potentially be drawn on any one of the three geometries, but it turns out they are never Euclidean.  In general, such a swollen triangle will no longer tesselate a plane, since fitting a whole number of these triangles will in general require an integer multiple of $2\pi$ radians; instead, these triangles fit around a given point in what could perhaps best be thought of as a $k$-layered corkscrew, either hyperbolic or spherical in kind.  In order to understand this configuration better, we are led to ask: given a hyperbolic triangle $T(a,b,c)$, how many conjugate triangles does it have up to isometry, how many are hyperbolic (versus spherical), and how can these numbers be efficiently computed?  

This naive ``swelling'' procedure can be reformulated in terms of the symmetry group of the tessellation as follows.  The reflections in the edges of $T$ generate a discrete group of isometries of the hyperbolic plane, and its orientation-preserving subgroup is a \defi{triangle group} $\Delta=\Delta(a,b,c)$.  Takeuchi \cite{Takeuchi} showed that the subalgebra $A \subseteq \M_2(\R)$ generated by $\Delta$ over its totally real invariant trace field $E$ is a quaternion algebra over $E$ (for precise definitions, see section 2).  Let $r$ be the number of split real places of $A$.  Then there is an embedding $\Delta \hookrightarrow \PSL_2(\R)^r$ as a discrete subgroup, and so $\Delta$ acts properly by orientation-preserving isometries on $\calH^r$, where $\calH$ is the hyperbolic plane.  The quotient $Y(\Delta)=\Delta \backslash \calH^r$ is a complex orbifold of dimension $r$.  Accordingly, we define the \defi{arithmetic dimension} $\adim(a,b,c)=\adim(\Delta(a,b,c))$ to be the number of split real places of $A$, and if $\adim(a,b,c)=r$, we say that $(a,b,c)$ is \defi{$r$-arithmetic}.  

The preceding definition is the natural way to make the above intuition rigorous, as it keeps track of the action of ``swelling'' on the symmetries of the tessellation---and, at the same time, it is the right notion for arithmetic applications.  An $r$-arithmetic triangle group yields an arithmetic lattice $\Delta \hookrightarrow \PSL_2(\R)^r$, and such lattices are central objects of study in many areas of mathematics.  The quotient $Y(\Delta)$ is the complex points of a \emph{quaternionic Shimura variety}, studied in detail by Cohen--Wolfart \cite{CohenWolfart}.  This area of research has seen significant renewed interest \cite{ClarkVoight,Kucharczyk,SchallerWolfart} in part because of its connection with the theory of \Belyi\ maps.  For these reasons, it is natural to seek out those triples with small arithmetic dimension $r$ in order to carry out further explicit investigations.  For example, a $2$-arithmetic triangle group yields a quaternionic Shimura surface \cite{Lai} equipped with a non-special rational curve lying on it, a configuration whose arithmetic geometry merits further study.  

Takeuchi \cite{Takeuchi,Takeuchi2} has determined all triples $(a,b,c)$ with $\adim(a,b,c)=1$; the corresponding triangle groups $\Delta(a,b,c)$ are then called \defi{arithmetic}.  Takeuchi finds precisely $85$ such triples, and they fall into $19$ commensurability classes.  
This finiteness result makes intuitive sense: if $a,b,c \in \Z_{\geq 2}$ are all large (with correspondingly small angles), we should be able to find a small enough swelling factor $k$ that keeps the triangle hyperbolic.  While Takeuchi follows this principle, his proof is not as simple, and carrying out the complete computation is a nontrivial task.  

The main result of our paper is to streamline Takeuchi's proof and to generalize his result to arbitrary arithmetic dimension, as follows.  For $r \in \Z_{\geq 1}$, let
\[ \calT(r)=\{(a,b,c) : \text{$a,b,c \in \Z_{\geq 2} \cup \{\infty\}$ with $\adim(a,b,c) = r$} \} \]
be the set of $r$-arithmetic triples.  

\begin{thm*}
The set $\calT(r)$ is finite.  Moreover, there exists an explicit algorithm that takes as input an integer $r \geq 1$ and produces as output the set $\calT(r)$ using $O(r^{14}\log^{21} r)$ bit operations.
\end{thm*}

In particular, we explicitly determine in this paper the group $H \subseteq (\Z/2m\Z)^\times$ giving triangles isometric to $T$ (i.e.\ acting trivially, Theorem \ref{thm:mult}) and we note that
\[ \adim(\Delta)=\#\left\{k \in (\Z/2m\Z)^\times/H : \kappa(a,b,c;k)<0 \right\} \]
where  
\[ 1-\kappa(a,b,c;k) = \cos^2\frac{k\pi}{a} + \cos^2\frac{k\pi}{b} + \cos^2\frac{k\pi}{c} + 2\cos\frac{k\pi}{a}\cos\frac{k\pi}{b}\cos\frac{k\pi}{c}. \] 
We provide a simple formula (Lemma \ref{lem:range}) that detects the sign of $\kappa(a,b,c;k)$ using exact arithmetic.  The proof of our theorem combines bits from algebra, Galois theory, number theory, as well as some analytic estimates.  In addition to the main theorem, in section 5 we also exhibit algorithms that perform very well in practice.  The list of triples in $\calT(r)$ for $r \leq 5$ is provided in section 6, along with the cardinalities $\#\calT(r)$ for $r \leq 15$.  

The authors would like to thank Robert Kucharczyk, Carl Pomerance, Dan Rockmore, and the anonymous referee for their helpful comments and suggestions, as well as Kayla Horak for her initial contributions to this project.  The second author was supported by an NSF CAREER Award (DMS-1151047).

\section{Background and notation}

In this section, we set up some basic background and notation on triangle groups and quaternion algebras.  For further reference, see Takeuchi \cite{Takeuchi0} and Clark--Voight \cite[\S 5]{ClarkVoight}, and the references therein. 

Let $a,b,c \in \Z_{\geq 2} \cup \{\infty\}$.  Without loss of generality, we may assume that $a \leq b \leq c$.  Suppose further that
\[ \chi(a,b,c)=\frac{1}{a}+\frac{1}{b}+\frac{1}{c}-1<0. \]
In particular, we have $b,c \geq 3$.  Then there is a triangle $T=T(a,b,c)$ with angles $\pi/a,\pi/b,\pi/c$ in the hyperbolic plane, unique up to isometry.  The reflections in the sides of this triangle generate a discrete group; its orientation-preserving subgroup is the \defi{triangle group} $\Delta(a,b,c) \leq \PSL_2(\R)$, a Fuchsian group with presentation
\begin{equation} 
\Delta=\Delta(a,b,c)=\langle \delta_a, \delta_b, \delta_c \mid \delta_a^a=\delta_b^b=\delta_c^c=\delta_a\delta_b\delta_c=1 \rangle; 
\end{equation}
when one of $s=a,b,c$ has $s=\infty$, we interpret $\delta_s^s=1$ to be a trivial relation, so for example, 
\[ \Delta(\infty,\infty,\infty) = \la \delta_a, \delta_b, \delta_c \mid \delta_a\delta_b\delta_c = 1 \rangle \]
is isomorphic to the free group on two generators.

For $s \in \Z_{\geq 2}$, let $\zeta_s=\exp(2\pi i/s)$ and $\lambda_s=2\cos(2\pi/s)=\zeta_s+\zeta_s^{-1}$, and by convention let $\zeta_\infty=1$ and $\lambda_\infty=2$.  By the half-angle formula, $\lambda_{2a}^2=\lambda_a+2$.  Let $\Delta^{(2)} \leq \Delta$ be the subgroup generated by $-1$ and $\delta^2$ for $\delta \in \Delta$.  Then $\Delta^{(2)} \trianglelefteq \Delta$ is a normal subgroup with quotient $\Delta/\Delta^{(2)}$ an elementary abelian $2$-group.  

We define the \defi{invariant trace field} of $\Delta(a,b,c)$ to be the totally real number field
\begin{equation} 
E=\Q(\Tr \Delta^{(2)})=\{\Tr(\delta^2) : \delta \in \Delta\})=\Q(\lambda_{a},\lambda_{b},\lambda_{c},\lambda_{2a}\lambda_{2b}\lambda_{2c})
\end{equation}
(well-defined as $\Tr(\delta^2)$ is independent of the lift of $\delta \in \Delta$ to $\SL_2(\R)$).  Takeuchi \cite{Takeuchi} showed that the subalgebra $A \subseteq \M_2(\R)$ generated over $E$ by the preimage of $\Delta$ in $\SL_2(\R)$ is a quaternion algebra over $E$, and indeed
\begin{equation} \label{eqn:Aquat}
A \simeq \quat{\lambda_b^2-4, (\lambda_b+2)(\lambda_c+2)\beta}{E} 
\end{equation}
where
\begin{equation} \label{eqn:beta}
\beta = \lambda_a + \lambda_b + \lambda_c + \lambda_{2a}\lambda_{2b}\lambda_{2c} + 2 = \lambda_{2a}^2+\lambda_{2b}^2+\lambda_{2c}^2+\lambda_{2a}\lambda_{2b}\lambda_{2c}-4 \neq 0.
\end{equation}

Let $m=\lcm(\{a,b,c\} - \{\infty\})$.  Then the invariant trace field $E$ sits in the following field diagram:
\begin{equation} \label{eqn:fielddiagram}
\begin{aligned} 
\xymatrix{
K = \Q(\zeta_{2a},\zeta_{2b},\zeta_{2c}) \ar@{-}^{H_2}[dr] \ar@{-}@/_4pc/_{G \simeq (\Z/2m\Z)^\times}[ddddr] 
\ar@{-}@/_1pc/_{H_1}[dddr]
\ar@{-}@/^3pc/^H[ddrr]\\
&F_2 = \Q(\lambda_{2a},\lambda_{2b},\lambda_{2c}) \ar@{-}[dr] \ar@{-}[dd] \\
&&E = \Q(\lambda_{a},\lambda_{b},\lambda_{c},\lambda_{2a}\lambda_{2b}\lambda_{2c}) \ar@{-}[dl] \\
& F_1 = \Q(\lambda_{a},\lambda_{b},\lambda_{c}) \ar@{-}[d] \\
&\Q
} 
\end{aligned}
\end{equation}
We identify 
\begin{align*}
(\Z/2m\Z)^\times &\xrightarrow{\sim} G=\Gal(K/\Q) \\
k &\mapsto \sigma_k \qquad \text{where $\sigma_k(\zeta_{2m})=\zeta_{2m}^k$}.
\end{align*}  
Accordingly, the real (infinite) places of $E$ are indexed by classes $k \in (\Z/2m\Z)^\times/H$, with
\[ \sigma_k(\lambda_{s})=\sigma_k(2\cos(2\pi/s)) = \sigma_k(\zeta_{s}+\zeta_{s}^{-1})=\zeta_s^k+\zeta_s^{-k}=2\cos(2k\pi/s) \]
for $s=a,b,c$.  At an embedding $E \hookrightarrow \R$, corresponding to the class of $\sigma_k$ in $G/H$, we have either $\sigma_k(\beta)>0$ or $\sigma_k(\beta)<0$, and accordingly we have $A \hookrightarrow \M_2(\R)$ (A is \defi{split}) or $A \hookrightarrow \HH$ (A is \defi{ramified}) where $\HH$ is the division ring of real Hamiltonians. 

Let $\Z_E$ be the ring of integers of $E$ and let $\Lambda \subseteq A$ be the $\Z_E$-order of $A$ generated by $\Delta^{(2)}$.  Then \cite[Proposition 5.11]{ClarkVoight} there is an embedding $\Delta \hookrightarrow N_A(\Lambda)/E^\times$, where $N_A(\Lambda)$ is the normalizer of $\Lambda$ in $A$.  

\begin{defn}
The \defi{arithmetic dimension} $\adim(a,b,c)$ of the triple $(a,b,c)$ is the number of split real (infinite) places of $A$.  If $\adim(a,b,c)=r$, we say that $(a,b,c)$ is \defi{$r$-arithmetic}.
\end{defn}

Takeuchi calls 1-arithmetic triples \defi{arithmetic}~\cite{Takeuchi}.  

Let $r=\adim(a,b,c)$.  Then by definition we have an embedding $A \hookrightarrow A \otimes_{\Q} \R \simeq \M_2(\R)^r \times \HH^{[E:\Q]-r}$ where $\HH \simeq \quat{-1,-1}{\R}$ is the division ring of Hamiltonians over $\R$.  Projecting onto the factor $\M_2(\R)^r$ we obtain an embedding $\Delta \hookrightarrow \PSL_2(\R)^r$ as a discrete subgroup and $\Delta$ acts properly on $\calH^r$ by orientation-preserving isometries, where $\calH$ is the hyperbolic plane.  The quotient $Y(\Delta)=\Delta \backslash \calH^r$ has the structure of a complex orbifold of dimension $r$; it is compact if and only if $\infty \not\in \{a,b,c\}$.  In this way, the study of $r$-arithmetic triangle groups is the study of certain natural arithmetic lattices in $\PSL_2(\R)^r$.

The arithmetic dimension is given in Galois-theoretic terms as follows.  For $k \in (\Z/2m\Z)^\times$ we define the \defi{curvature}
\begin{equation} \label{eqn:kappa_curvature}
\kappa(a,b,c;k) = 1 - \left(\cos^2\frac{k\pi}{a} + \cos^2\frac{k\pi}{b} + \cos^2\frac{k\pi}{c} + 2\cos\frac{k\pi}{a}\cos\frac{k\pi}{b}\cos\frac{k\pi}{c}\right) 
\end{equation}
so that from \eqref{eqn:beta} we have
\begin{equation} \label{eqn:sigmakappa}
\sigma_k(\beta)=-4\kappa(a,b,c;k). 
\end{equation}

\begin{lem} \label{lem:adimkappa}
We have
\[ \adim(a,b,c)=\{ k \in (\Z/2m\Z)^\times/H : \sigma_k(\beta)>0\}=\{ k \in (\Z/2m\Z)^\times/H : \kappa(a,b,c;k)<0\}. \]
\end{lem}

\begin{proof}
We refer to \eqref{eqn:Aquat}.  At the real place $\sigma_k$ with $k \in (\Z/2m\Z)^\times/H$, we have 
\[ A_k=A \otimes_{E,\sigma_k} \R \simeq \quat{\sigma_k(\lambda_b)^2-4, (\sigma_k(\lambda_b)+2)(\sigma_k(\lambda_c)+2)\sigma_k(\beta)}{\R}. \]
Since $|\sigma_k(\lambda_b)|=2\lvert\cos(2k\pi/b)\rvert<2$, we have $\sigma_k(\lambda_b)^2-4<0$ (since $b \geq 3$).  Similarly, $\sigma_k((\lambda_b+2)(\lambda_c+2))>0$, so $A$ is split at the place $\sigma_k$ if and only if $\sigma_k(\beta)>0$.  The result then follows from \eqref{eqn:sigmakappa}.
\end{proof}

In section 2, we compute the size of the group $H$ (multiplicity) and in section 3 we investigate explicitly the curvature $\kappa(a,b,c;k)$.

\begin{rmk}
We conclude this section with the connection made in the introduction to the intuition of ``swelling'' triangles.  
We identify the acute triangle $T(a,b,c)$ with the cosines of its angles
\[ (\cos \pi/a, \cos \pi/b, \cos \pi/c)=\frac{1}{2}(\lambda_{2a},\lambda_{2b},\lambda_{2c}) \]
and so the conjugate triangle, with angles swollen by a factor $k$, has angles with cosines
\[ (\cos k\pi/a, \cos k\pi/b, \cos k\pi/c)=\frac{1}{2}\sigma_k(\lambda_{2a},\lambda_{2b},\lambda_{2c}). \]

If the quaternion algebra $A$ is split at $\sigma_k$, then under this embedding we have a Galois conjugate embedding $\Delta \hookrightarrow \PSL_2(\R)$, and the fixed points of $\delta_a,\delta_b,\delta_c$ form a hyperbolic triangle with angles having the above cosines; on the other hand, if $A$ is ramified at $\sigma_k$, then $\Delta \hookrightarrow \HH^1/\{\pm 1\}$ where $\HH^1$ is a compact group isomorphic to the $3$-sphere, and we similarly obtain a spherical triangle.  

Together with Lemma \ref{lem:adimkappa}, this justifies the name \emph{curvature} to the function $\kappa$: when $\kappa(a,b,c;k)>0$ the ambient space is the sphere (positive curvature), and when $\kappa(a,b,c;k)<0$ the ambient space is the hyperbolic plane (negative curvature).  (And although this case does not arise for us, the case $\kappa(a,b,c;k)=0$ would correspond to the flat Euclidean plane.)

In this way, we capture the intuition of examining the effect of swelling the angles of a hyperbolic triangle.  (Counting the corresponding triangles with the multiplicity coming from the subgroup $H$, as opposed to the smaller subgroup $\Gal(K/F)$, corresponds to the additional condition imposed by the product relation $\delta_a\delta_b\delta_c=1$ in the triangle group $\Delta$.)
\end{rmk}

\section{Multiplicity}

Our first task is to understand the multiplicity given by the size of the groups $H_2 \leq H \leq H_1$ in the field diagram \eqref{eqn:fielddiagram}, using Galois theory.  $[H_1:H] \leq 2$, so we first compute $\#H_1$ and then decide if $H_1=H$ or not.  We retain the notation from the previous section.

\begin{thm} \label{thm:mult}
Let $a,b,c \in \Z_{\geq 2}$.  Let $t$ be the number of coprime pairs among $a,b,c$ and let $u$ be the number of pairs whose gcd is $\leq 2$.  We have $\#H_2=\max(2,2^t)$.  
\begin{enumalph}
\item If $a,b,c$ are all odd, then $\#H_1=\max(2,2^u)$ and $H=H_1$.  
\item Suppose that at least one of $a,b,c$ is even.  Then $\#H_1=2\max(2,2^u)$, and $H=H_1$ if and only if one of the following holds, for some permutation of $a,b,c$:
\begin{enumroman}
\item[\textup{(i)}] $u\leq 1$ and $\ord_2(a)=\ord_2(b)> \ord_2(c)$, or
\item[\textup{(ii)}] $u=2$ and $\gcd(a,b)=\gcd(a,c)=1$ and $\ord_2(b)=\ord_2(c)>0$.
\end{enumroman}
\end{enumalph}
\end{thm}

\begin{proof}
By Galois theory, 
\[ H_2 \simeq \{k\in (\Z/2m\Z)^\times : \cos(k\pi/s) = \cos(\pi/s) \text{ for $s=a,b,c$}\}. \]
We have $\cos(k\pi/s)=\cos(\pi/s)$ if and only if $k \equiv \pm 1 \pmod{2s}$, so $k \in (\Z/2m\Z)^\times$ has $k \in H_2$ if and only if $k$ is a solution to the simultaneous congruences
\begin{align*}
k&\equiv \pm1\pmod{2a}\\
&\equiv \pm1\pmod{2b}\\
&\equiv \pm1\pmod{2c}
\end{align*}
where the signs may be chosen arbitrarily (among eight possibilities).  Choosing these signs consistently, we always have the solution $k \equiv \pm 1 \pmod{2m}$, so $\#H_2 \geq 2$.  More generally, for $\epsilon \in \{1,-1\}$, by the Chinese Remainder Theorem there exists a $k\in \Z/2m\Z$ that satisfies
\begin{align*}
k&\equiv +\epsilon \pmod{2a}\\
&\equiv +\epsilon\pmod{2b}\\
&\equiv -\epsilon\pmod{2c}
\end{align*}
if and only if $\epsilon \equiv - \epsilon \pmod{\gcd(2a,2c)}$ and $\epsilon \equiv - \epsilon \pmod{\gcd(2b,2c)}$, which is the case if and only if $\gcd(a,b)=\gcd(b,c)=1$.  Necessarily, such a solution has $k \in (\Z/2m\Z)^\times$.  Therefore, after permuting $a,b,c$ and choosing $\epsilon$, we see that if $t=1$ then $\#H_2=2$, if $t=2$ then $\#H_2=4$ (one extra pair of solutions), and if $t=3$ then $\#H_2=8$ (every choice of signs is possible).  

A similar argument holds for $H_1$, where now $\cos^2(k\pi/s)=\cos^2(\pi/s)$ if and only if $k \equiv \pm 1 \pmod{s}$ and we solve
\begin{align*}
k&\equiv \pm1\pmod{a}\\
&\equiv \pm1\pmod{b}\\
&\equiv \pm1\pmod{c}.
\end{align*}
We find $\max(2,2^u)$ solutions modulo $m$ by the Chinese remainder theorem, as in the previous paragraph, and each such solution has $k \in (\Z/m\Z)^\times$.  For solutions modulo $2m$, we note that when $m$ is odd, the map $(\Z/2m\Z)^\times \to (\Z/m\Z)^\times$ is a bijection, but when $m$ is even, this map is surjective with kernel of size $2$.  Therefore, $\#H_1=\max(2,2^u)$ if $a,b,c$ are all odd, and $\#H_1=2\max(2,2^u)$ otherwise.  

We now turn to $H \leq H_1$.  For $\sigma_k \in H_1$, $\sigma_k \in H$ if and only if 
\[ \cos\frac{k\pi}{a}\cos\frac{k\pi}{b}\cos\frac{k\pi}{c}=\cos\frac{\pi}{a}\cos\frac{\pi}{b}\cos\frac{\pi}{c} \]
or equivalently, $\cos(k\pi/s)=-\cos(\pi/s)$ for an \emph{even} number of $s=a,b,c$.  Therefore, $H \neq H_1$ if and only if there exists a solution $k \in (\Z/2m\Z)^\times$ to the either the system of congruences
\begin{equation} \tag{I} \label{eqn:ka1}
\begin{aligned}
k&\equiv a\pm1\pmod {2a}\\
&\equiv \pm1\pmod {2b}\\
&\equiv \pm1\pmod {2c},
\end{aligned} 
\end{equation}
up to permutation of $a,b,c$, or the system of congruences
\begin{equation} \tag{II} \label{eqn:ka1b1c1}
\begin{aligned}
k&\equiv a\pm1\pmod {2a}\\
&\equiv b\pm1\pmod {2b}\\
&\equiv c\pm1\pmod {2c}.
\end{aligned} 
\end{equation}

We may now finish the proof of statement (a).  If all of $a,b,c$ are odd, there can be no solution to \eqref{eqn:ka1}--\eqref{eqn:ka1b1c1} since $k \equiv a \pm 1 \pmod{2a}$ implies that $k$ is even; and we conclude that $H=H_1$.

We henceforth assume that at least one of $a,b,c$ is even.  Then a solution $k \in \Z/2m\Z$ of these congruences automatically has $k \in (\Z/2m\Z)^\times$.  We consider first the congruences \eqref{eqn:ka1}; since $k$ is odd, $a$ is even.  We have 3 cases.  

\subsubsection*{Case \textup{(\ref{eqn:ka1}a)}}  First, by the Chinese Remainder Theorem, there exists a solution to
\begin{equation} \label{eqn:CONGRUENCE1}
\begin{aligned}
k&\equiv a+\epsilon\pmod{2a}\\
&\equiv \epsilon\pmod{2b}\\
&\equiv \epsilon\pmod{2c},
\end{aligned}
\end{equation}
if and only if $a+\epsilon\equiv \epsilon\pmod{2\gcd(a,s)}$ for both $s=b,c$.  We have $a+\epsilon \equiv \epsilon\pmod{2\gcd(a,s)}$ if and only if $2\gcd(a,s) \mid a$ if and only if $\ord_2(a)>\ord_2(s)$, and thus \eqref{eqn:CONGRUENCE1} has a solution if and only if 
\begin{equation} \label{eqn:hh1.1}
\ord_2(a)>\ord_2(b),\ord_2(c). 
\end{equation}

\subsubsection*{Case \textup{(\ref{eqn:ka1}b)}}  Second, there exists a solution to
\begin{equation} \label{eqn:CONGRUENCE2}
\begin{aligned}
k &\equiv a+\epsilon\pmod{2a}\\
&\equiv -\epsilon\pmod{2b}\\
&\equiv -\epsilon\pmod{2c},
\end{aligned}
\end{equation}
if and only if $a+\epsilon\equiv -\epsilon\pmod{2\gcd(a,s)}$ for $s=b,c$.  Recalling $a$ is even,
\begin{align*}
&a+\epsilon\equiv -\epsilon\pmod{2\gcd(a,s)} \\
&\qquad\Leftrightarrow 2\gcd(a,s) \mid (a+2\epsilon) \\
&\qquad\Leftrightarrow \text{$\gcd(a,s) \leq 2$ and ($\gcd(a,s)=2 \Rightarrow \ord_a(2)=1$)}
\end{align*}
So \eqref{eqn:CONGRUENCE2} has a solution if and only if 
\begin{equation} \label{eqn:hh1.2}
\text{$\gcd(a,b),\gcd(a,c) \leq 2$ and ($\gcd(a,b)=2$ or $\gcd(a,c)=2$ $\Rightarrow$ $\ord_2(a)=1$).}
\end{equation}

\subsubsection*{Case \textup{(\ref{eqn:ka1}c)}}  Third, by similar arguments, there exists a solution to
\begin{equation} \label{eqn:CONGRUENCE3}
\begin{aligned}
k&\equiv a+\epsilon\pmod{2a}\\
&\equiv \epsilon\pmod{2b}\\
&\equiv -\epsilon\pmod{2c},
\end{aligned}
\end{equation}
if and only if 
\begin{equation} \label{eqn:hh1.3}
\begin{gathered}
\text{$\ord_2(a) > \ord_2(b)$ and $\gcd(a,c) \leq 2$ and $\gcd(b,c)=1$ and} \\
\text{($\gcd(a,c)=2$ $\Rightarrow$ $\ord_2(a)=1$).}
\end{gathered}
\end{equation}

Finally, we consider the congruences \eqref{eqn:ka1b1c1}.  Then we must have $a,b,c$ all even, and we have two cases.  

\subsubsection*{Case \textup{(\ref{eqn:ka1b1c1}a)}}  
First, there exists a solution to
\begin{equation} \label{eqn:CONGRUENCE4}
\begin{aligned}
k&\equiv a+\epsilon\pmod{2a}\\
&\equiv b+\epsilon\pmod{2b}\\
&\equiv c+\epsilon\pmod{2c},
\end{aligned}
\end{equation}
if and only if $a+\epsilon\equiv b+\epsilon\pmod{2\gcd(a,b)}$ and symmetrically for all three pairs.  Since
\begin{align*}
a+\epsilon\equiv b+\epsilon\pmod{2\gcd(a,b)} &\Leftrightarrow a-b\equiv 0\pmod{2\gcd(a,b)},\\
&\Leftrightarrow a/\!\gcd(a,b) - b/\!\gcd(a,b) \textup{ is even}\\
&\Leftrightarrow \ord_2(a) = \ord_2(b).
\end{align*}
The congruences \eqref{eqn:CONGRUENCE4} have a solution if and only if 
\begin{equation} \label{eqn:hh1.4}
\ord_2(a) = \ord_2(b) = \ord_2(c). 
\end{equation}

\subsubsection*{Case \textup{(\ref{eqn:ka1b1c1}b)}}  Second, there exists a solution to
\begin{equation} \label{eqn:CONGRUENCE5}
\begin{aligned}
k&\equiv a-\epsilon\pmod{2a}\\
&\equiv b+\epsilon\pmod{2b}\\
&\equiv c+\epsilon\pmod{2c},
\end{aligned}
\end{equation}
if and only if $a-\epsilon\equiv s+\epsilon\pmod{2\gcd(a,s)}$ for both $s=b,c$ and as in the previous case $b+\epsilon\equiv c+\epsilon\pmod{2\gcd(b,c)}$, so $\ord_2(b)=\ord_2(c)$.  Since all of $a,b,c$ are even,
\begin{align*}
&a-\epsilon\equiv s+\epsilon\pmod{2\gcd(a,s)} \\
&\qquad\Leftrightarrow a-s\equiv 2\epsilon\pmod{2\gcd(a,s)},\\
&\qquad\Leftrightarrow \gcd(a,s)=2  \text{ and } a-s\equiv 2\pmod{4}.
\end{align*}
Thus the congruences have a solution if and only if
\begin{equation} \label{eqn:hh1.5}
\begin{gathered}
\text{$\gcd(a,b)=\gcd(a,c)=2$ and} \\
\text{ ($\ord_2(a)>\ord_2(b)=\ord_2(c)=1$ or $\ord_2(a)=1<\ord_2(b)=\ord_2(c)$).}
\end{gathered}
\end{equation}

We have shown that $H \neq H_1$ if and only if for some permutation of $a,b,c$, one of the conditions \eqref{eqn:hh1.1}, \eqref{eqn:hh1.2}, \eqref{eqn:hh1.3}, \eqref{eqn:hh1.4}, or \eqref{eqn:hh1.5} holds.  What remains is to simplify these by splitting these into cases.

So suppose $H=H_1$.

\subsubsection*{Case $u=3$} Suppose $u=3$.  Then by \eqref{eqn:hh1.1} none of $a,b,c$ are divisible by $4$ (and still at least one of $a,b,c$ is even).  If all of $a,b,c$ are even, then \eqref{eqn:hh1.4}, a contradiction.  If say $a,c$ are even and $b$ is odd, then \eqref{eqn:hh1.3} holds, a contradiction.  If say $a$ is even and $b,c$ are odd, then \eqref{eqn:hh1.1} holds, again a contradiction.  We conclude that $H \neq H_1$ in this case.  

\subsubsection*{Case $u=2$} Suppose $u=2$, with $\gcd(a,b),\gcd(a,c) \leq 2$ but $\gcd(b,c)>2$.  Suppose say $\gcd(a,b)=2$; then by \eqref{eqn:hh1.2}, we have $\ord_2(a)>1=\ord_2(b)=\ord_2(c)$, and this contradicts \eqref{eqn:hh1.5}.  Thus $\gcd(a,b)=\gcd(a,c)=1$.  If $a$ is even, then $b,c$ are odd, and this contradicts \eqref{eqn:hh1.1}, so $a$ is odd, and without loss of generality $b$ is even, and moreover \eqref{eqn:hh1.1} implies that $\ord_2(b)=\ord_2(c)>0$.  With these hypotheses as in (ii), each of these five conditions fail, so conversely we have $H=H_1$.

\subsubsection*{Case $u \leq 1$}  If $u \leq 1$, then by \eqref{eqn:hh1.1} and \eqref{eqn:hh1.4}, the maximum $\max(\ord_2(a),\ord_2(b),\ord_2(c))$ must be achieved by exactly two values, say $\ord_2(a)=\ord_2(b)>\ord_2(c)$.  But then already the five conditions fail, so again we have $H=H_1$, and the proof is complete.  
\end{proof}

The case in which $a$, $b$, or $c$ is infinite can be proven similarly.  

\begin{thm} \label{thm:multoo}
Let $a,b \in \Z_{\geq 2} \cup \{\infty\}$ and $c=\infty$.  Let $e \geq 1$ be the number of $s=a,b,c$ such that $s=\infty$.  Then the following hold.
\begin{enumalph}
\item If $e=3$, then $H_1=H=H_2$ is the trivial group.  
\item If $e=2$, then $H_1=H=H_2$ is the group of order $2$.
\item Suppose $e=1$.  Let $g=\gcd(a,b)$.  Then 
\[ \#H_1=\begin{cases}
8, \text{ if $g\leq2$ and $a$ or $b$ is even;} \\
2, \text{ if $g>2$ and $a$ and $b$ are odd;}\\
4, \text{ otherwise;}
\end{cases} \]
and
\[ \#H_2=\begin{cases}
4, \text{ if $g=1$;} \\
2, \text{ otherwise;}
\end{cases} \]
and $H=H_2$ if and only if $\ord_2(a) = \ord_2(b)$ and $\gcd(a,b)\neq2$.
\end{enumalph}
\end{thm}

\begin{proof}
We may assume without loss of generality that $a\leq b \leq c$.  For part (a), when $e=3$ and $a=b=c=\infty$, we cannot multiply by any $k$ and thus only have one triple.

So we turn to (b), and we suppose $e=2$.  Thus $a$ is finite and $b=c=\infty$.  Then the size of both $H$ and $H_2$ is the number of $k\in (\Z/2a\Z)^\times$ such that
$$k\equiv\pm1\pmod{2a},$$
and the size of $H_1$ is the number of $k\in (\Z/2a\Z)^\times$ such that
$$k\equiv\pm1\pmod{a}.$$
Both congruences have only two solutions: $k=\pm 1$, so $\#H_1=\#H=\#H_2 = 2$.

So we are left with the case $e=1$; we suppose $a$ and $b$ are finite and $c=\infty$.  The proof is similar to and simpler than that of Theorem \ref{thm:mult}, so we are brief.  We find $\#H_2$ in a similar manner to in Theorem \ref{thm:mult}, solving the congruences
\begin{align*}
k&\equiv \pm1\pmod{2a}\\
&\equiv \pm1\pmod{2b}
\end{align*}
We have two trivial solutions, and two additional solutions if and only if $\gcd(a,b)=1$.  Likewise, we find $\#H_1$ by solving the congruences
\begin{align*}
k&\equiv \pm1\pmod{a}\\
&\equiv \pm1\pmod{b},
\end{align*}
obtaining two trivial solutions, and two additional solutions if and only if $\gcd(a,b)\leq 2$.  But as in the proof of Theorem \ref{thm:mult}, we must double this value if either of $a,b$ is even.

$H\neq H_1$ if and only if, for some permutation of $a,b$, there exists a solution $k \in (\Z/2m\Z)^\times$ to the simultaneous congruences
\begin{align*}
k&\equiv a\pm1\pmod{2a}\\
&\equiv \pm1\pmod{2b}.
\end{align*}
Since $k$ is odd, we have no solution if $a$ is odd.  Hence, if $a$ and $b$ are odd then $H=H_1$.  

Suppose that $a$ is even.  The congruences have a solution if and only if $a\pm1 \equiv \pm1\pmod{\gcd(2a,2b)$.}  Let $\epsilon \in \{\pm1\}$.  Then
\begin{align*}
a+\epsilon \equiv \epsilon\pmod{\gcd(2a,2b)} &\Leftrightarrow a \equiv 0 \pmod{2\gcd(a,b)}\\
&\Leftrightarrow \ord_2(a)>\ord_2(b).
\end{align*}
and
\begin{align*}
a+\epsilon \equiv -\epsilon\pmod{\gcd(2a,2b)} &\Leftrightarrow a\pm2 \equiv 0 \pmod{2\gcd(a,b)}\\
&\Leftrightarrow \gcd(a,b)=1 \textup{ or } (\gcd(a,b)=2\textup{ and }\ord_2(a)=1).
\end{align*}

So $H\neq H_1$ if and only if, for some permutation of $a,b$, $a$ is even and either (i) $\ord_2(a) > \ord_2(b)$ or (ii) $\gcd(a,b)=1 \textup{ or } (\gcd(a,b)=2\textup{ and }\ord_2(a)=1)$.  Note that if $a$ is even and $\gcd(a,b)=1$, then $\ord_2(a) > \ord_2(b)$.  Also, if $\ord_2(a) \leq \ord_2(b)$, then $\gcd(a,b)=2$ implies that $\ord_2(a)=1$.  Furthermore, note that $(\ord_2(a) > \ord_2(b)$ or $\gcd(a,b)=2)$ implies that $a$ is even.  Therefore, if one of $a,b$ is even, then $H\neq H_1$ if and only if $\ord_2(a) > \ord_2(b)$ or $\gcd(a,b)=2$ or  $\ord_2(b) > \ord_2(a)$.  Putting it all together, $H=H_2$ if and only if $\ord_2(a) = \ord_2(b)$ and $\gcd(a,b)\neq2$.  Note that this also covers the case in which $a,b$ are both odd.
\end{proof}

\begin{cor}\label{multthm} 
There exists an algorithm that takes as input $a,b,c \in \Z_{\geq 2} \cup \{\infty\}$ and produces as output $\#H$ using $O(\log^2 m)$ bit operations.
\end{cor}

\begin{proof}
Applying either Theorem \ref{thm:mult} or Theorem \ref{thm:multoo}, we need only apply a constant number of applications of $\gcd$ or $\ord_2$.  We can compute $\gcd(a,b)$ using $O(\log^2(\max(a,b)))$ bit operations, and $\ord_2(a)$ can be implemented in $O(\log(a))$ time by finding the number of trailing zeros in the binary representation of $a$.  Since $a,b,c \leq m$, the result follows.
\end{proof}

\section{Curvature}

In this section, we discuss a method of determining the sign of the curvature $\kappa(a,b,c;k)$ defined in \eqref{eqn:kappa_curvature}.  Our main result is that there is an easy exact calculation that determines this sign; this characterization was essentially given by Takeuchi~\cite{Takeuchi}.

We continue with our assumption that $a,b,c \in \Z_{\geq 2} \cup \{\infty\}$ satisfy $a \leq b \leq c$ and $\chi(a,b,c)<0$.  In particular, $\kappa(a,b,c;1)<0$.  To avoid potential confusion with notation, we treat the case $c=\infty$ right away.

\begin{lem} \label{lem:koo}
If $c=\infty$, then $\kappa(a,b,c;k) < 0$ for all $k \in (\Z/2m\Z)^\times$.
\end{lem}

\begin{proof}
If $c=\infty$, then $\cos(k \pi/c)=1$ so
\[	\kappa(a,b,c;k) 
	= 1-\left(\cos^2\frac{k\pi}{a}+\cos^2\frac{k\pi}{b}+1+2\cos^2\frac{k\pi}{a}\cos\frac{k\pi}{b}\right)\\
	= -\left(\cos\frac{k\pi}{a}+\cos\frac{k\pi}{b}\right)^2 \leq 0 \]
and equality holds if and only if $\cos(k\pi/a)=-\cos(k\pi/b)$; elementary arguments show that this holds if and only if $a=b=2$, and the triple $(2,2,\infty)$ is Euclidean, not hyperbolic.
\end{proof}

We therefore have the following computation of arithmetic dimension in a special case.

\begin{cor} \label{cor:adimoooo}
We have $\adim(\infty,\infty,\infty)=\adim(2,\infty,\infty)=1$, and for $a \geq 3$, $\adim(a,\infty,\infty)=\adim(a,a,\infty)=\phi(2a)/2=\phi(a)/\!\gcd(2,a)$.  
\end{cor}

\begin{proof}
According to the field diagram \eqref{eqn:fielddiagram}, when either (1) $b=c=\infty$ or (2) $a=b$ and $c=\infty$, we have $m=a$ and $E=F=\Q(\lambda_{2a})$, so $H=\{\pm 1\} \leq (\Z/2a\Z)^\times=G$.  By Lemma \ref{lem:koo}, the arithmetic dimension is equal to $\#G/H=\phi(2a)/2$.
\end{proof}

For $s \in \{a,b,c\}$ with $s \neq \infty$, let $k_{s}$ denote the unique integer with $0\leq k_{s} \leq s$ such that 
\begin{equation} \label{eqn:defofks}
\cos\frac{k_{s}\pi}{s}=\cos\frac{k\pi}{s};
\end{equation}
we have $k_s=\lvert{k'\rvert}$ where $k' \in [-s,s]$ and $k \equiv k' \pmod{2s}$.  If $s=\infty$, we let $k_s=k$.  	Given $k$ and $s \neq \infty$, we can compute $k_s$ in a straightforward manner using $O(\log k \log s)$ bit operations.
		

Next, we have $\kappa(a,b,c;k) \neq 0$ for all $k \in (\Z/2m\Z)^\times$, since $\kappa(a,b,c;k)=\sigma_k(\kappa(a,b,c;1))$ and $\kappa(a,b,c;1) \neq 0$.  

We will now prove an important lemma that determines the sign of $\kappa(a,b,c;k)$; this result can be extracted from the work of Takeuchi \cite[(18)]{Takeuchi}.  

\begin{lem}\label{lem:range} 
Suppose that $a,b,c \in \Z_{\geq 2} \cup \{\infty\}$ and let $k \in (\Z/2m\Z)^\times$.  Then $\kappa(a,b,c;k) \leq 0$ if and only if
\begin{equation}  \label{eqn:rangeeqn}
\left\lvert\frac{k_{a}}{a} + \frac{k_{b}}{b}-1\right\rvert \leq \frac{k_c}{c} \leq 1-\left\lvert\frac{k_{a}}{a} - \frac{k_{b}}{b}\right\rvert, 
\end{equation}
and $\kappa(a,b,c;k)=0$ if and only if one of the two equalities holds in \eqref{eqn:rangeeqn}.
\end{lem}

	\begin{proof}
	We write $\kappa(a,b,c;k)$ as a quadratic function $z=\cos(k_{c}\pi/c)$, with $z=1$ if $c=\infty$: we have
	$$\kappa(a,b,c;k) = f(z) = -(z^2+tz+n),$$ where
	\[ t = 2\cos\frac{k_{a}\pi}{a}\cos\frac{k_{b}\pi}{b}, \qquad n=\cos^2\frac{k_{a}\pi}{a}+\cos^2\frac{k_{b}\pi}{b}-1 \]
from \eqref{eqn:defofks}.  Then by the quadratic formula, $f(z) \geq 0$ if and only if 
	$$\frac{-t-\sqrt{t^2 - 4n}}{2} \leq \cos\frac{k_{c}\pi}{c} \leq \frac{-t+\sqrt{t^2 - 4n}}{2}$$
and $f(z)=0$ if and only if one of the two equalities holds.  The discriminant simplifies as
	\begin{align*}
	\sqrt{t^2-4n}
	&=  \sqrt{4\cos^2\frac{k_{a}\pi}{a}\cos^2\frac{k_{b}\pi}{b}-4\cos^2\frac{k_{a}\pi}{a}-4\cos^2\frac{k_{b}\pi}{b}-4} \\
	&= 2\sqrt{\left(\cos^2\frac{k_{a}\pi}{a}-1\right)\left(\cos^2\frac{k_{b}\pi}{b}-1\right)}
 = 2\sin\frac{k_{a}\pi}{a}\sin\frac{k_{b}\pi}{b},
	\end{align*}
	so we have	
	\begin{align*}
	\frac{-t-\sqrt{t^2 - 4n}}{2} &= -\cos\frac{k_{a}\pi}{a}\cos\frac{k_{b}\pi}{b} -  \sin\frac{k_{a}\pi}{a}\sin\frac{k_{b}\pi}{b} \\
	&= -\cos\left\lvert\frac{k_{a}\pi}{a} - \frac{k_{b}\pi}{b}\right\rvert = \cos\left(\pi - \left\lvert\frac{k_{a}\pi}{a} - \frac{k_{b}\pi}{b}\right\rvert\right)
	\end{align*}
	and similarly
	\begin{align*}
	\frac{-t+\sqrt{t^2 - 4n}}{2} = -\cos\left(\frac{k_{a}\pi}{a} + \frac{k_{b}\pi}{b}\right)=  \cos\left\lvert\frac{k_{a}\pi}{a} + \frac{k_{b}\pi}{b} - \pi\right\rvert.
	\end{align*}
	Therefore, $\kappa(a,b,c;k)\leq 0$ if and only if 
	$$\cos \left[\left(1-\left\lvert\frac{k_{a}}{a} - \frac{k_{b}}{b}\right\rvert\right)\pi\right] \leq \cos\frac{k_{c}\pi}{c} \leq \cos\left(\left\lvert\frac{k_{a}}{a} + \frac{k_{b}}{b}-1\right\rvert\pi\right) $$
	and $\kappa(a,b,c;k)=0$ if and only if one of the equalities holds.
	Note that $0 \leq k_c\pi/c \leq \pi$.  Also, since $k_a<a$ and $k_b<b$,
	$$\left\lvert\frac{k_{a}}{a} - \frac{k_{b}}{b}\right\rvert < 1 \quad \text{and} \quad \left\lvert\frac{k_{a}}{a} + \frac{k_{b}}{b} - 1\right\rvert < 1.$$
Then, since $\cos(x)$ is decreasing between $x=0$ and $x=\pi$, the result follows.
	\end{proof}
		
	\begin{cor} 
	\label{ineqs}
	If $\kappa(a,b,c;k)>0$, then $k_a/a + k_b/b + k_c/c > 1$.
	\end{cor}
		
	\begin{proof}
	If $\kappa(a,b,c;k)>0$, we know from Lemma \ref{lem:range} that
	\[ -\left(\frac{k_{a}}{a} + \frac{k_{b}}{b}-1\right) < \frac{k_{c}}{c}, \]
	from which the corollary directly follows.\end{proof}

Lemma \ref{lem:range} gives an exact algorithm for computing $\kappa(a,b,c;k)$ and $\adim(a,b,c)$ as follows, and in particular we do not need to estimate the error term in evaluating the cosine.

\begin{prop} \label{prop:computeadim}
There exists an algorithm that takes as input $a,b,c \in \Z_{\geq 2} \cup \{\infty\}$ and produces as output $\adim(a,b,c)$ using $O(m\log^2 m)$ bit operations.
\end{prop}

\begin{proof}
We may assume $a \leq b \leq c$.  We loop over the elements $k \in \Z/2m\Z$.  We can compute $\gcd(k,2m)$ using $O(\log^2 m)$ bit operations, and so keep only $k \in (\Z/2m\Z)^\times$.  If $c=\infty$, by Lemma \ref{lem:koo} we have automatically $\kappa(a,b,c;k)<0$, so we can simply count.  Otherwise, using integer arithmetic, $\kappa(a,b,c;k)<0$ if and only if
\begin{equation} \label{eqn:CURVATURE}
c\left\lvert k_{a}b + k_{b}a-ab\right\rvert < k_{c}ab < c\left(ab-\left\lvert k_{a}b - k_{b}a\right\rvert\right).
\end{equation}
This check requires $O(\log^2 c)$ bit operations when $c \neq \infty$, since $a,b \leq c$.  Finally, by Corollary \ref{multthm}, we can compute the multiplicity $\#H$ using $O(\log^2 m)$ bit operations (and this need only be done once for the triple).  We then return our count divided by $\#H$.
\end{proof}

We will also make use of the following lemma.

	\begin{lem} \label{lem:abkab}
	Let $(a,b,c)$ be a hyperbolic triple, and let $k\in (\Z/2m\Z)^\times$.  Then 
	\[ |ab-k_{a}b-k_{b}a| \geq 1. \]
	\end{lem}
	
	\begin{proof}
		Since $ab-k_{a}b-k_{b}a \in \Z$, we need only show it is nonzero; for the purpose of contradiction, assume $ab-k_{a}b-k_{b}a=0$.
		Then $k_{b}a \equiv ka \equiv 0\pmod b$; but $k$ is coprime to $2m$ and hence $b$, so $b \mid a$.  Similarly, $a \mid b$.  So $a=b$ and $k_a=k_b$.  Then $a^2-2k_{a}a = 0$, so $a=2k_{a}$ and $2k \equiv 0\pmod a$, and so since $k$ is relatively prime to $a$, we have that $a=b=2$.  But $\chi(2,2,c) \geq 0$ for all $c$, so $(a,b,c)$ is not hyperbolic, a contradiction.
	\end{proof}

\section{Finiteness of arithmetic triples of bounded dimension}

For $r \in \Z_{\geq 1}$, let
\[ \calT(r)=\{(a,b,c) : \text{$a,b,c \in \Z_{\geq 2} \cup \{\infty\}$ with $a \leq b \leq c$ and $\adim(a,b,c) = r$} \} \]
be the set of $r$-arithmetic triples.  The main result of this section is the following theorem.

\begin{thm} \label{main}
For all $r \in \Z_{\geq 1}$, the set $\calT(r)$ is finite.  
\end{thm}

Let $r \in \Z_{\geq 1}$.  We will examine hyperbolic conjugate triangles indexed by primes, so we make the following definition.

\begin{defn}
A prime $q$ is the \defi{$r$th nondividing prime} of $n \in \Z_{\geq 1}$ if $q \nmid n$ and there are exactly $r-1$ primes $p < q$ such that $p \nmid n$.
\end{defn}

For example, the first nondividing prime of $n$ is the smallest prime $q \nmid n$.  

For $a,b,c \in \Z_{\geq 2} \cup \{\infty\}$, let $\NDP(a,b,c;r)$ denote the $r$th nondividing prime of $2m=2\lcm(\{a,b,c\} \smallsetminus \{\infty\})$.  

Following the strategy of Takeuchi, our first step is to show that if $(a,b,c) \in \calT(r)$ is $r$-arithmetic, then $a,b,c$ are bounded above in terms of $\NDP(a,b,c;r)$: that is, for a given prime $q$, there are only finitely many $r$-arithmetic triples $(a,b,c)$ such that $\NDP(a,b,c;r)=q$.  Our second step will then be to show that there are no $r$-arithmetic triples $(a,b,c)$ such that $\NDP(a,b,c;r)=q$ for $q$ large enough (depending on $r$).  Hence the number of $r$-arithmetic triples is finite, concluding the proof of Theorem \ref{main}.

The role of the $r$th nondividing prime in the proof is made evident from the following lemma.

\begin{lem} \label{lem:abcs}
Let $s=\max(\{a,b,c\}-\{\infty\})$.  Let $q=\NDP(a,b,c;r)$ and suppose $s > 2q$.  Then there are $r$ distinct primes $p_1,\dots,p_r \leq q$ such that the $r+1$ elements $1,p_1,\dots,p_r \in (\Z/2m\Z)^{\times}/H$ are all distinct.
\end{lem}

\begin{proof}
By definition, there are $r$ distinct primes $p_1,\dots,p_r \leq q < s/2$ with each $p_i$ coprime to $2m$, so the elements $1,p_1,\dots,p_r$ are distinct in $(\Z/s\Z)^{\times}/\{\pm 1\}$.  The map
\[ (\Z/2m\Z)^\times/H \simeq \Gal(E/\Q) \to \Gal(\Q(\lambda_s)/\Q) \simeq (\Z/s\Z)^{\times}/\{\pm 1\} \]
is surjective, so the classes $1,p_1,\dots,p_r \in (\Z/2m\Z)^\times/H$ are distinct.
\end{proof}

We break the argument into two cases: let
\begin{equation} \label{eqn:cnc}
\begin{aligned} 
\calT(r)_{0} &=\{(a,b,c) \in \calT(r) : a,b,c \in \Z\} \\
\calT(r)_{\infty}&=\{(a,b,\infty) \in \calT(r)\}
\end{aligned}
\end{equation}
so that $\calT(r)=\calT(r)_{0} \sqcup \calT(r)_{\infty}$.  

We first treat $\calT(r)_{\infty}$, and we call this the \defi{noncompact} case (as in section 1, the quotient $Y(\Delta)=\Delta \backslash \calH^r$ is noncompact).  We computed the arithmetic dimension for signatures $(a,\infty,\infty)$ already in Corollary \ref{cor:adimoooo}, but we include them here for completeness.

\begin{prop} \label{prop6inf}
Let $(a,b,\infty) \in \calT(r)_\infty$ and let $q = \NDP(a,b,c;r)$.  If $b \neq \infty$, then $b \leq 2q$; and if $a \neq \infty$, then $a \leq 2q$.  
\end{prop}

\begin{proof}
By Lemma \ref{lem:koo}, $\kappa(a,b,\infty;k)<0$ for all $k \in (\Z/2m\Z)^{\times}$.  Suppose $b \neq \infty$ and assume that $2q<b$.  Then by Lemma \ref{lem:abcs}, $\kappa(a,b,\infty;p_i)<0$ for primes $p_i$ such that $1,p_1,\dots,p_r \in (\Z/2m\Z)^{\times}/H$ are distinct, so $(a,b,c) \not\in \calT(r)$, a contradiction.  Thus $b \leq 2q$.  The same argument works in the simpler case $b=\infty$ and $a \neq \infty$.  
\end{proof}
	
	We will now prove bounds on $a,b,c$ in terms of $q = \NDP(a,b,c;r)$ (cf.\ Takeuchi {\cite[Proposition 6]{Takeuchi}}) in the \defi{compact} case, where $a,b,c \neq \infty$.
		
	\begin{prop}\label{prop6}
	Let $(a,b,c) \in \calT(r)_{0}$ and let $q = \NDP(a,b,c;r)$.  Then $a<3q$ and $b<6q^2$ and $c<18q^4$.
\end{prop}

	\begin{proof}
	
	If $c<2q$, the proposition immediately follows.  Therefore let us assume that $2q<c$.
		
	Note that, since $a,b,c$ are relatively prime to $q$, we have that $q_{a}<a$, $q_{b}<b$, and $q_{c}<c$.
	Since $2q<c \neq \infty$, by Lemma \ref{lem:abcs} there are primes $p_1,\dots,p_r$ such that $1,p_1,\dots,p_r \in (\Z/2m\Z)^\times/H$ are distinct.  Since $(a,b,c) \in \calT(r)$, this implies that $\kappa(a,b,c;p_i)>0$ for some prime $p=p_i \leq q$.  Therefore, the inequalities of Lemma \ref{lem:range} and Corollary \ref{ineqs} hold with $k=p$, and in particular,
	\begin{equation}  \label{eq:keyineq}
	\frac{p_a}{a}+\frac{p_b}{b}+\frac{p_c}{c}>1.
	\end{equation}
If $a>2q$, then $2p \leq 2q<a\leq b\leq c$, so $p_{a} = p_{b}=p_{c}=p \leq q$, and the inequality \eqref{eq:keyineq} becomes $1<q(1/a+1/b+1/c)$.  In particular, $1 \leq 3q/a$ so $a < 3q$.  By inequality \eqref{eq:keyineq}
\[ 2\left(\frac{q}{b}\right) \geq \frac{q}{b} + \frac{q}{c} \geq \frac{q_b}{b} + \frac{q_c}{c} > 1-\frac{q_a}{a} \geq \frac{1}{a}. \] 
Therefore, $b<2qa<6q^2$.

Since, by Lemma \ref{lem:range},
\[ \left|\frac{q_{a}}{a} + \frac{q_{b}}{b}-1\right| < \frac{q_{c}}{c}, \] 
we have that 
\[ \frac{q}{c} \geq \frac{q_{c}}{c} > \frac{|ab - q_{a}b-q_{b}a|}{ab}. \]  
By Lemma \ref{lem:abkab}, $ab \neq q_{a}b + q_{b}a$,  so 
\[ \frac{q}{c} > \frac{1}{ab}. \]  
Therefore, $c < qab < 18q^4$.
\end{proof}

Let 
\begin{equation} \label{eqn:Trq}
\calT(r;q)=\{(a,b,c) \in \calT(r) : \NDP(a,b,c;r)=q\}. 
\end{equation}

\begin{cor} \label{cor:finitetrq}
For any prime $q$, $\#\calT(r;q)<\infty$.
\end{cor}

\begin{proof}
Immediate from the inequalities in Propositions \ref{prop6inf} and \ref{prop6}, bounding $a,b,c$ in terms of $q$ in each case.
\end{proof}

We now proceed with the second step.  First, a slightly ugly but key lemma. 
Let $p_{i}$ denote the $i$th prime, so $p_1=2$, $p_2=3$, and so on.

\begin{lem} \label{p6lem} Let $r,j \in \Z_{\geq 1}$ with $j-r \geq 3$.  If 
\[ p_{1}\cdots p_{j-r}<648p_{j}^7, \] 
then
\[ j < \max(48,2r) \quad\textup{and}\quad p_j < \max(251,5r\log r). \]
\end{lem}

		\begin{proof}
	Suppose $p_{1}\cdots p_{j-r}<648p_{j}^7$.  We will use some basic inequalities on primes, referring to work of Rosser--Schoenfeld \cite{Rosser} as a convenient reference.
	
	Let $\theta(x)=\sum_{p \leq x}\log p = \log(\prod_{p \leq x} p)$ be Chebyshev's $\theta$-function.  Then by assumption
	\begin{equation} \label{eqn:theta}
	 \theta(p_{j-r}) = \log(p_1 \cdots p_{j-r}) < \log(648p_j^7) = \log(648)+7\log(p_j).
	 \end{equation}
	
	We need two claims.
	
	\begin{claim} \label{eqn:claim1}
	If $j \geq r+16$, then $2j-2r<\theta(p_{j-r})$.
	\end{claim}
	
	\begin{proof}
	We have $\theta(x)>x(1-1/\log x)$ for $x \geq 41=p_{13}$ \cite[(3.16)]{Rosser}.  Also, $p_k \geq k\log k$ for all $k \geq 1$ \cite[(3.12)]{Rosser}.  Putting these together,
	\[ \theta(p_k) = \sum_{p \leq p_k} \log p_k > p_k\left(1-\frac{1}{\log p_k}\right) \geq k\log k\left(1-\frac{1}{\log(k\log k)}\right) \]
	for $k \geq 13$.  
	(Stronger inequalities of this type are known, e.g., work of Robin \cite[p.~376]{Robin}.)
We deduce that $\theta(p_{k}) > 2k$ for $k\geq 16$.  Therefore, when $j \geq r+16$, we have $2j-2r<\theta(p_{j-r})$.  
    \end{proof}

	For $k \geq 6$ we have \cite[(3.13)]{Rosser} 
	\begin{equation} \label{eqn:dusart}
	p_{k} \leq k(\log k+\log\log k) \leq 2k\log k
	\end{equation} 
	(see also Dusart \cite[Lemma 1]{Dusart}).
    
	\begin{claim} \label{eqn:claim2}
	If $j \geq 48$, then $\log(648)+7\log(p_j) < j$.
	\end{claim}
	
	\begin{proof}
	Apply \eqref{eqn:dusart} to get $$\log(648)+7\log(p_j)< \log(648)+7\log(2j(\log j)) < j$$
	when $j \geq 48$.
	\end{proof}
		
	Now we apply Claims \ref{eqn:claim1} and \ref{eqn:claim2} which bound the left- and right-hand sides of \eqref{eqn:theta} in terms of $j,r$.  When $j \geq \max(48,r+16)$, the inequality
    \[ 2j-2r < \theta(p_{j-r}) < \log(648)+7\log(p_j) < j \]
    holds, and consequently $j < 2r$.  Therefore, in all cases 
    \[ j < \max(48,r+16,2r)=\max(48,2r) \]
    as claimed.

	To conclude, we prove the final inequality.  We have shown that $j<\max(48,2r)$.  This yields two cases.  If $j < 48$, then by \eqref{eqn:dusart} we have $p_j < 96(\log 48+\log\log 48) \leq 251$ or $j \leq 6$ and then $p_j \leq 13 \leq 251$ as well.  Otherwise, we have $j\geq 48$ and $j<2r$, so $r \geq 24$, and again by \eqref{eqn:dusart} we have 
	\[ p_j<4r\log(2r) \leq r \log 16 + 4r\log r \leq r \log r + 4r\log r  = 5r \log r. \] 
Combining these two cases, we find $p_j < \max(251,5r \log r)$ as claimed.
\end{proof}


\begin{prop} \label{prop7}
$\calT(r;q)=\emptyset$ for $q \geq \max(251,5r \log r)$.
\end{prop}

\begin{proof}
As in the preceding lemma, let $p_{i}$ denote the $i$th prime.  Let $n$ such that $p_{n}=q$.
	
	All but $r-1$ \emph{odd} primes less than $q$ divide $m=\lcm(\{a,b,c\} \smallsetminus \{\infty\})$.  Let $P$ be the product of these $n-r$ primes.  Then $P \mid m$, so $P \leq m$.  Hence 
	 $$p_{1}\cdots p_{n-r} \leq 2P \leq 2 \prod_{\substack{s\in\{a,b,c\} \\ s \neq \infty}}{s}.$$
	Therefore, by Propositions \ref{prop6inf} and \ref{prop6}, 
	\[ p_{1}\cdots p_{n-r}<2(3p_n)(6p_n^2)(18p_n^4)=648p_n^7. \]  So by Lemma \ref{p6lem}, $p_n=q \leq \max(251,5r\log r)$.  
\end{proof}

We conclude with the proof of Theorem \ref{main}.

\begin{proof}[Proof of Theorem $\ref{main}$]
We have $\calT(r)=\bigsqcup_q \calT(r;q)$.  By Proposition \ref{prop7}, this is a finite union; and by Corollary \ref{cor:finitetrq}, each $\calT(r;q)$ is finite.  Therefore $\calT(r)$ is finite.
\end{proof}

\section{Algorithms}

In this chapter, we will present algorithms for enumerating $r$-arithmetic triangle groups.  

With the explicit bounds exhibited in the previous section, the following theorem follows directly.

\begin{thm}
There exists an explicit algorithm that takes as input $r \in \Z_{\geq 1}$ and produces as output the set $\calT(r)$ that runs using $O(r^{14}\log^{21}r)$ bit operations.
\end{thm}

\begin{proof}
Recall from \eqref{eqn:Trq} that we write $T(r) = \bigsqcup_q \calT(r;q)$ as a disjoint union, sorting triples $(a,b,c) \in \calT(r)$ by their $r$th least nondividing prime $q=\NDP(a,b,c;r)$.  By Proposition \ref{prop7}, $\calT(r;q)=\emptyset$ for $q > \max(251,5r\log r)$.  The triples with $q\leq 251$ can be enumerated and checked in constant time, so we need only analyze the cases in which $251<q \leq 5r\log r$.

As in \eqref{eqn:cnc}, we compute two cases.  In the noncompact case (with $c=\infty$), by Proposition \ref{prop6inf}, if $(a,b,\infty) \in \calT(r;q)$ then $a,b=\infty$ or $a,b \leq 2q$, so we must loop over $O(q^2)$ triples.  In the compact case, by Proposition \ref{prop6}, if $(a,b,c) \in \calT(r;q)$ then we must loop over
\[ (3q)(6q^2)(18q^4)=324q^7=324(5r\log r)^7 = O(r^7 \log^7 r) \]
triples.  For each such triple $(a,b,c)$, we can compute $\adim(a,b,c)$ using $O(m\log^2 m)$ bit operations by Proposition \ref{prop:computeadim}, and $m=O(abc)=O(r^7 \log^7 r)$.  So the algorithm runs using $O(r^{14}\log^{21}r)$ bit operations.
\end{proof}

The above enumerative procedure gives a ``brute force'' algorithm to compute $\calT(r)$ explicitly.  In the remainder of this section, we describe an algorithm which performs better in practice.

\begin{thm}\label{mainalg}
There exists an algorithm that takes as input integer $r$, where $r\geq 1$, and produces as output the set of $r$-arithmetic triples $(a,b,c)$, where $a,b,c\in\Z_{\geq2}\cup\infty$.
\end{thm}

The algorithm in this Theorem is provided by the algorithm FIND\textunderscore R\textunderscore ARITHMETIC below; its proof of correctness is also given below.  

Using the sieve of Eratosthenes, we can maintain a global list ODD\textunderscore PRIMES of the odd prime numbers in increasing order.  We can start with a finite number of primes, and extend as necessary.

\begin{algorithm}[H]
\caption*{\textbf{Algorithm} FIND\textunderscore R\textunderscore ARITHMETIC(\texttt{r})}
\begin{algorithmic}[1]
\REQUIRE $r \in \Z_{\geq 1}$
\ENSURE Outputs the set of $r$-arithmetic triangle groups $(a,b,c)$,  where $a,b,c \in \Z_{\geq 2}\cup \infty$.
\STATE Initialize $\texttt{arithmetic}$ to empty list
\STATE $\texttt{maxNDPIndex} \leftarrow \max(36,2\texttt{r}+5)$ //find maximum $r$th non-dividing prime of $(a,b,c)$
\STATE $\texttt{maxNDP}\leftarrow \textup{ODD\textunderscore PRIMES}[\texttt{maxNDPIndex}]$
\STATE $\texttt{end} \leftarrow [3*\texttt{maxNDP},6*\texttt{maxNDP}^2,18*\texttt{maxNDP}^4]$ //upper bounds on $a,b,c$
\STATE $\texttt{maxPrimeIndex} \leftarrow \texttt{maxNDPIndex}*\texttt{r}$ //scale with $r$ to make sure each $divisor$ is big enough
\FOR{$\texttt{a} = 2$ to $\texttt{end}[1]$}
\FOR{$\texttt{b} = \max(\texttt{a},3)$ to $\texttt{end}[2]$}
\STATE MAP\textunderscore BOUNDS\textunderscore TO\textunderscore PRIMES($\texttt{a},\texttt{b},\texttt{maxPrimeIndex},\texttt{end}[3]$)
\STATE SEARCH\textunderscore Cs($\texttt{a},\texttt{b},\texttt{end}[3])$
\ENDFOR
\ENDFOR
\STATE // Now for the noncompact case:
\FOR {$\texttt{a}=2$ to $2*\texttt{maxNDP}$}
\IF {IS\textunderscore R\textunderscore ARITHMETIC($\texttt{a},\infty,\infty,\texttt{r}$)}
\STATE Add ($\texttt{a},\texttt{b},\texttt{c}$) to $\texttt{arithmetic}$
\ENDIF
\FOR {$\texttt{b}=\texttt{a}$ to $2*\texttt{maxNDP}$}
\IF {IS\textunderscore R\textunderscore ARITHMETIC($\texttt{a},\texttt{b},\infty,r$)}
\STATE Add ($\texttt{a},\texttt{b},\texttt{c}$) to $\texttt{arithmetic}$
\ENDIF
\ENDFOR
\ENDFOR
\STATE output $\texttt{arithmetic}$
\end{algorithmic}
\end{algorithm}

\begin{algorithm}[H]
\caption*{\textbf{Subroutine} MAP\textunderscore BOUNDS\textunderscore TO\textunderscore PRIMES($\texttt{a},\texttt{b},\texttt{maxPrimeIndex},\texttt{maxc}$)}
\begin{algorithmic}[1]
\ENSURE Returns map \texttt{boundToPrimes} of each bound on $c$ to the primes associated to that bound.  That is, if $c > \texttt{bound}$, then $c<2q$ or $q|c$ for all $q \in$\texttt{boundToPrimes}[$\texttt{bound}$].

\STATE Initialize \texttt{boundToPrimes} to an empty map
\FOR{$\texttt{j}=0$ to $\texttt{maxPrimeIndex}$}
\STATE $\texttt{q} \leftarrow$ ODD\textunderscore PRIMES[j]
\IF{$\texttt{q}$ does not divide $\texttt{a}$ or $\texttt{b}$}
\STATE $\texttt{bound} \leftarrow \lceil \texttt{q}*\texttt{a}*\texttt{b}/|\texttt{q}_{\textup{a}}*\texttt{q} + \texttt{q}_{\textup{b}}*\texttt{a}-\texttt{a}*\texttt{b}|\rceil$
\IF{\texttt{boundToPrimes}[$\texttt{bound}$] == NULL}
\STATE Initialize \texttt{boundToPrimes}[$\texttt{bound}$] to empty list
\ENDIF
\STATE Add $\texttt{q}$ to \texttt{boundToPrimes}[$\texttt{bound}$]
\ENDIF
\ENDFOR
\STATE $\texttt{boundToPrimes}[\texttt{maxc}] = \emptyset$
\RETURN \texttt{boundToPrimes}
\end{algorithmic}
\end{algorithm}

\begin{algorithm}[H]
\caption*{\textbf{Subroutine} SEARCH\textunderscore Cs($a,b,\texttt{maxc}$)}
\begin{algorithmic}[1]
\ENSURE Adds all $r$-arithmetic triangle groups $(a,b,c)$ to $\texttt{arithmetic}$ for given $a$ and $b$, where $c\leq \texttt{maxc}$.
\FOR{$\texttt{c} = \texttt{b}$ to ODD\textunderscore PRIMES[\texttt{maxNDPIndex}]}
\IF {IS\textunderscore R\textunderscore ARITHMETIC($\texttt{a},\texttt{b},\texttt{c},\texttt{r}$)}
\STATE Add ($\texttt{a},\texttt{b},\texttt{c}$) to $\texttt{arithmetic}$
\ENDIF
\ENDFOR
\STATE $\texttt{divisors}$ $\leftarrow$ integer array of length $n$
\STATE for $\texttt{j}=1$ to $\texttt{r}$, $\texttt{divisors}$[$j$] $\leftarrow$ 1
\STATE $\texttt{i}\leftarrow 1$ // index in $\texttt{divisors}$
\STATE $\texttt{startc}\leftarrow \texttt{b}$
\FOR{\texttt{bound} in the key set of \texttt{boundToPrimes}, sorted by increasing order}
\FOR{$\texttt{divisor}$ in $\texttt{divisors}$}
\STATE CHECK\textunderscore MULTIPLES($\texttt{divisor}, \texttt{a},\texttt{b}, \texttt{bound}, \texttt{startc}, \texttt{r})$
\ENDFOR
\FOR {$\texttt{q}$ in $\texttt{boundToPrimes}[\texttt{bound}]$}
\STATE $\texttt{divisors}$[$\texttt{i}$] $\leftarrow$ $\texttt{divisors}$[$\texttt{i}$]*$\texttt{q}$
\STATE $\texttt{i} \leftarrow (\texttt{i}+1)\text{ mod }\texttt{r}$
\ENDFOR
\IF{$\texttt{divisor} >\texttt{maxc}$  $\forall \texttt{divisor} \in \texttt{divisors}$}
\STATE break
\ENDIF
\STATE $\texttt{startc} \leftarrow \textup{max}(\texttt{startc},\texttt{bound})$
\ENDFOR
\end{algorithmic}
\end{algorithm}


\begin{algorithm}[H]
\caption*{\textbf{Subroutine} CHECK\textunderscore MULTIPLES($\texttt{divisor}, \texttt{a},\texttt{b}, \texttt{bound}, \texttt{startc}$)}
\begin{algorithmic}[1]
\ENSURE Checks $(a,b,c)$, where $c$ is each multiple of \texttt{divisor} between $\texttt{startc}$ and $\texttt{bound}$.  Adds $r$-arithmetic triples to $\texttt{arithmetic}$.
\STATE $\texttt{c} \leftarrow \texttt{startc}-(\texttt{startc}$ mod $ \texttt{divisor})$
\IF{$\texttt{c}<\texttt{startc}$}
\STATE $\texttt{c} \leftarrow \texttt{c}+\texttt{divisor}$
\ENDIF
\WHILE{$\texttt{c}<\texttt{bound}$}
\IF {IS\textunderscore ARITHMETIC($\texttt{a},\texttt{b},\texttt{c},1$)}
\STATE Add ($\texttt{a},\texttt{b},\texttt{c}$) to $\texttt{arithmetic}$
\ENDIF
\STATE $\texttt{c} \leftarrow \texttt{c} + \texttt{divisor}$
\ENDWHILE
\end{algorithmic}
\end{algorithm}

\begin{algorithm}[H]
\caption*{\textbf{Subroutine} IS\textunderscore R\textunderscore ARITHMETIC($\texttt{a},\texttt{b},\texttt{c},\texttt{r}$)}
\begin{algorithmic}[1]
\ENSURE Returns whether $(a,b,c)$ is an $r$-arithmetic triple.
\STATE $\texttt{multiplicity} \leftarrow $GET\textunderscore MULTIPLICITY$(\texttt{a},\texttt{b},\texttt{c})$
\STATE $\texttt{m} \leftarrow 2*\lcm\{s\in\{\texttt{a},\texttt{b},\texttt{c}\}|s \;\textup{finite}\}$
\STATE $\texttt{numHyperbolic} \leftarrow 0$
\FOR{$\texttt{k} = 1$ to $\texttt{m}$}
\IF {gcd$(\texttt{m},\texttt{k}) == 1$}
\STATE $\texttt{sign} \leftarrow $CURVATURE$(\texttt{a},\texttt{b},\texttt{c},\texttt{k})$
\IF{$\texttt{sign} < 0$}
\STATE $\texttt{numHyperbolic} \leftarrow \texttt{numHyperbolic} + 1$
\ENDIF
\ENDIF
\ENDFOR
\RETURN $\texttt{numHyperbolic} == \texttt{r}*\texttt{multiplicity}$
\end{algorithmic}
\end{algorithm}

\begin{proof}[{Proof of Theorem \textup{\ref{mainalg}}}]
The correctness of \texttt{maxNDPIndex} follows from Proposition \ref{prop7}, and the correctness total bounds on $a$, $b$, and $c$ (stored in the array $\texttt{end}$) follows from Proposition \ref{prop6}.  Furthermore, the correctness of the noncompact case follows directly from Proposition \ref{prop6inf}.

Just checking all $(a,b,c)$ up to these bounds would give us a theoretically finite computation producing all arithmetic triples.  But since the bound on $c$ is so large, we do some additional ``filtering" in the SEARCH\textunderscore Cs subroutine.  This is a crucial improvement in efficiency.

So it suffices to prove the correctness of SEARCH\textunderscore Cs.  We have from Lemma \ref{lem:range} that for any arithmetic triple $(a,b,c)$ and $k<2c$ coprime to $a,b,c$, 
	$$ \left|\frac{k_{a}}{a} + \frac{k_{b}}{b}-1\right| < \frac{k_{c}}{c} < 1-\left|\frac{k_{a}}{a} - \frac{k_{b}}{b}\right|.$$

Then $$c < \frac{k_{c}}{\left\lvert\frac{k_{a}}{a} + \frac{k_{b}}{b}-1\right\rvert} \leq \frac{k}{\left\lvert\frac{k_{a}}{a} + \frac{k_{b}}{b}-1\right\rvert} = \frac{kab}{\left\lvert k_{a}b + k_{b}a-ab\right\rvert}.$$
Therefore, for any arithmetic triple $(a,b,c)$ with given $a$ and $b$, and any prime $q$, one of the following: 
\begin{center}
either $q \mid c$ or $c < 2q$ or $c < \displaystyle{\left\lceil \frac{qab}{\left\lvert q_{a}b + q_{b}a-ab\right\rvert}\right\rceil}$.
\end{center}

In SEARCH\textunderscore Cs, we first check all $c<2*\texttt{maxNDP}$, where $\texttt{maxNDP}$ is the maximum $r$th Non-Dividing Prime for any arithmetic triple $(a,b,c)$.  This is because we know that, for any arithmetic triple $(a,b,c)$, there is a prime $q\leq \texttt{maxNDP}$ that does not divide $c$, and so $c<2q\leq 2*\texttt{maxNDP}$ or $c<\lceil qab/|q_{a}b + q_{b}a-1|\rceil$, the bound associated to $q$.

Suppose $(a,b,c)$ is $r$-arithmetic, $c>2*\texttt{maxNDP}$, and $c>\texttt{bound}$, where $\texttt{bound}$ is in the keyset of \texttt{boundToPrimes}.  Then for all primes $q$ in \texttt{boundToPrimes}[$\texttt{bound}'$], where $\texttt{bound}'<\texttt{bound}$, we have that if $q$ does not divide $c$, then $\kappa(a,b,c;q) > 0$ and $\kappa(a,b,c;q) \neq \kappa(a,b,c;1)$.  Let $B$ = \{$q|q\in$\texttt{boundToPrimes}[$\texttt{bound}'$],$\texttt{bound}'\leq \texttt{bound}$\}.  Therefore, there exist at most $r-1$ primes $q\in B$ that do not divide $c$.  The algorithm partitions $B$ into $r$ sets, and lets each $\texttt{divisor}$ in $\texttt{divisors}$ be the product of primes in one such set.  Hence, $c$ must be a multiple of at least one $\texttt{divisor}$ in $\texttt{divisors}$.  Therefore, the algorithm checks all possible $r$-arithmetic triples ($a,b,c$).
\end{proof}

We ran this algorithm and obtained the results in listed in the final section.

\section{Arithmetic triples of small dimension}

\subsection{Lists of $r$-arithmetic triples for $r \leq 5$}
\subsubsection{$1$-arithmetic triples}\ \\

76 compact $1$-arithmetic triples:

\begin{multicols}{5}
(2, 3, 7)

(2, 3, 8)

(2, 3, 9)

(2, 3, 10)

(2, 3, 11)

(2, 3, 12)

(2, 3, 14)

(2, 3, 16)

(2, 3, 18)

(2, 3, 24)

(2, 3, 30)

(2, 4, 5)

(2, 4, 6)

(2, 4, 7)

(2, 4, 8)

(2, 4, 10)

(2, 4, 12)

(2, 4, 18)

(2, 5, 5)

(2, 5, 6)

(2, 5, 8)

(2, 5, 10)

(2, 5, 20)

(2, 5, 30)

(2, 6, 6)

(2, 6, 8)

(2, 6, 12)

(2, 7, 7)

(2, 7, 14)

(2, 8, 8)

(2, 8, 16)

(2, 9, 18)

(2, 10, 10)

(2, 12, 12)

(2, 12, 24)

(2, 15, 30)

(2, 18, 18)

(3, 3, 4)

(3, 3, 5)

(3, 3, 6)

(3, 3, 7)

(3, 3, 8)

(3, 3, 9)

(3, 3, 12)

(3, 3, 15)

(3, 4, 4)

(3, 4, 6)

(3, 4, 12)

(3, 5, 5)

(3, 6, 6)

(3, 6, 18)

(3, 8, 8)

(3, 8, 24)

(3, 10, 30)

(3, 12, 12)

(4, 4, 4)

(4, 4, 5)

(4, 4, 6)

(4, 4, 9)

(4, 5, 5)

(4, 6, 6)

(4, 8, 8)

(4, 16, 16)

(5, 5, 5)

(5, 5, 10)

(5, 5, 15)

(5, 10, 10)

(6, 6, 6)

(6, 12, 12)

(6, 24, 24)

(7, 7, 7)

(8, 8, 8)

(9, 9, 9)

(9, 18, 18)

(12, 12, 12)

(15, 15, 15)
\end{multicols}

9 noncompact $1$-arithmetic triples:

\begin{multicols}{5}
($\infty$, $\infty$, $\infty$)

(2, $\infty$, $\infty$)

(2, 3, $\infty$)

(2, 4, $\infty$)

(2, 6, $\infty$)

(3, $\infty$, $\infty$)

(3, 3, $\infty$)

(4, 4, $\infty$)

(6, 6, $\infty$)

\end{multicols}

This list of compact 1-arithmetic Triples agrees with that of Takeuchi, thus verifying his results \cite{Takeuchi}.

\subsubsection{$2$-arithmetic triples}\ \\

148 compact $2$-arithmetic triples:

\begin{multicols}{5}
(2, 3, 13)

(2, 3, 15)

(2, 3, 17)

(2, 3, 20)

(2, 3, 21)

(2, 3, 22)

(2, 3, 26)

(2, 3, 28)

(2, 3, 36)

(2, 3, 40)

(2, 3, 42)

(2, 3, 60)

(2, 4, 9)

(2, 4, 11)

(2, 4, 14)

(2, 4, 15)

(2, 4, 16)

(2, 4, 20)

(2, 4, 24)

(2, 4, 30)

(2, 4, 42)

(2, 5, 7)

(2, 5, 9)

(2, 5, 12)

(2, 5, 15)

(2, 5, 60)

(2, 6, 7)

(2, 6, 9)

(2, 6, 10)

(2, 6, 14)

(2, 6, 18)

(2, 6, 20)

(2, 6, 30)

(2, 7, 8)

(2, 7, 12)

(2, 7, 21)

(2, 7, 28)

(2, 8, 12)

(2, 8, 24)

(2, 9, 9)

(2, 9, 12)

(2, 9, 36)

(2, 10, 12)

(2, 10, 20)

(2, 10, 30)

(2, 11, 11)

(2, 11, 22)

(2, 13, 26)

(2, 14, 14)

(2, 14, 28)

(2, 15, 15)

(2, 16, 16)

(2, 18, 36)

(2, 20, 20)

(2, 20, 40)

(2, 21, 42)

(2, 24, 24)

(2, 30, 30)

(2, 30, 60)

(2, 42, 42)

(3, 3, 10)

(3, 3, 11)

(3, 3, 13)

(3, 3, 14)

(3, 3, 18)

(3, 3, 20)

(3, 3, 21)

(3, 3, 30)

(3, 4, 5)

(3, 4, 7)

(3, 4, 8)

(3, 4, 9)

(3, 4, 10)

(3, 4, 20)

(3, 4, 28)

(3, 4, 36)

(3, 5, 6)

(3, 5, 7)

(3, 5, 9)

(3, 5, 10)

(3, 5, 15)

(3, 7, 7)

(3, 7, 21)

(3, 9, 9)

(3, 10, 10)

(3, 12, 36)

(3, 14, 14)

(3, 14, 42)

(3, 18, 18)

(3, 20, 20)

(3, 20, 60)

(3, 30, 30)

(4, 4, 7)

(4, 4, 8)

(4, 4, 10)

(4, 4, 12)

(4, 4, 15)

(4, 4, 21)

(4, 5, 6)

(4, 5, 10)

(4, 5, 12)

(4, 6, 8)

(4, 6, 10)

(4, 6, 12)

(4, 7, 7)

(4, 8, 24)

(4, 10, 20)

(4, 12, 12)

(4, 24, 24)

(5, 5, 6)

(5, 5, 30)

(5, 6, 6)

(5, 6, 10)

(5, 8, 40)

(5, 12, 12)

(5, 20, 20)

(5, 30, 30)

(6, 6, 7)

(6, 6, 9)

(6, 6, 10)

(6, 6, 15)

(6, 7, 7)

(6, 8, 8)

(6, 9, 9)

(6, 10, 10)

(7, 7, 14)

(7, 14, 14)

(7, 14, 42)

(7, 28, 28)

(8, 8, 12)

(8, 16, 16)

(9, 9, 18)

(9, 36, 36)

(10, 10, 10)

(10, 10, 15)

(10, 20, 20)

(10, 40, 40)

(11, 11, 11)

(12, 24, 24)

(13, 13, 13)

(14, 14, 14)

(15, 30, 30)

(15, 60, 60)

(18, 18, 18)

(20, 20, 20)

(21, 21, 21)

(21, 42, 42)

(30, 30, 30)
\end{multicols}

16 noncompact $2$-arithmetic triples:

\begin{multicols}{5}

(2, 5, $\infty$)

(2, 8, $\infty$)

(2, 10, $\infty$)

(2, 12, $\infty$)

(3, 4, $\infty$)

(3, 5, $\infty$)

(3, 6, $\infty$)

(4, $\infty$, $\infty$)

(4, 6, $\infty$)

(4, 12, $\infty$)

(5, $\infty$, $\infty$)

(5, 5, $\infty$)

(6, $\infty$, $\infty$)

(8, 8, $\infty$)

(10, 10, $\infty$)

(12, 12, $\infty$)

\end{multicols}

\subsubsection{$3$-arithmetic triples}\ \\

111 compact $3$-arithmetic triples:

\begin{multicols}{5}
(2, 3, 19)

(2, 3, 23)

(2, 3, 27)

(2, 3, 32)

(2, 3, 34)

(2, 3, 38)

(2, 3, 48)

(2, 3, 50)

(2, 3, 54)

(2, 3, 66)

(2, 4, 13)

(2, 4, 22)

(2, 4, 26)

(2, 4, 28)

(2, 4, 36)

(2, 5, 14)

(2, 5, 16)

(2, 5, 18)

(2, 6, 11)

(2, 6, 15)

(2, 6, 16)

(2, 6, 24)

(2, 7, 9)

(2, 7, 10)

(2, 7, 42)

(2, 8, 9)

(2, 8, 10)

(2, 8, 18)

(2, 9, 10)

(2, 10, 15)

(2, 12, 36)

(2, 12, 48)

(2, 13, 13)

(2, 16, 32)

(2, 17, 34)

(2, 19, 38)

(2, 22, 22)

(2, 24, 48)

(2, 25, 50)

(2, 26, 26)

(2, 27, 54)

(2, 28, 28)

(2, 33, 66)

(2, 36, 36)

(3, 3, 16)

(3, 3, 17)

(3, 3, 19)

(3, 3, 24)

(3, 3, 25)

(3, 3, 27)

(3, 3, 33)

(3, 6, 8)

(3, 6, 10)

(3, 6, 12)

(3, 6, 30)

(3, 7, 9)

(3, 9, 18)

(3, 9, 27)

(3, 11, 11)

(3, 15, 15)

(3, 16, 16)

(3, 16, 48)

(3, 18, 54)

(3, 22, 66)

(3, 24, 24)

(4, 4, 11)

(4, 4, 13)

(4, 4, 14)

(4, 4, 18)

(4, 5, 20)

(4, 6, 36)

(4, 9, 9)

(4, 10, 10)

(4, 12, 18)

(4, 14, 28)

(4, 18, 18)

(5, 5, 7)

(5, 5, 8)

(5, 5, 9)

(5, 6, 30)

(5, 7, 7)

(5, 8, 8)

(5, 9, 9)

(5, 10, 30)

(5, 15, 15)

(6, 6, 8)

(6, 6, 12)

(6, 8, 24)

(6, 9, 18)

(6, 15, 30)

(6, 36, 36)

(6, 48, 48)

(7, 7, 21)

(8, 8, 9)

(8, 12, 24)

(8, 32, 32)

(10, 15, 30)

(11, 22, 22)

(12, 12, 18)

(12, 12, 24)

(12, 48, 48)

(13, 26, 26)

(14, 28, 28)

(16, 16, 16)

(17, 17, 17)

(18, 36, 36)

(19, 19, 19)

(24, 24, 24)

(25, 25, 25)

(27, 27, 27)

(33, 33, 33)
\end{multicols}

13 noncompact $3$-arithmetic triples:

\begin{multicols}{5}
(2, 7, $\infty$)

(2, 9, $\infty$)

(2, 14, $\infty$)

(2, 18, $\infty$)

(3, 7, $\infty$)

(3, 9, $\infty$)

(6, 18, $\infty$)

(7, $\infty$, $\infty$)

(7, 7, $\infty$)

(9, $\infty$, $\infty$)

(9, 9, $\infty$)

(14, 14, $\infty$)

(18, 18, $\infty$)

\end{multicols}

\subsubsection{$4$-arithmetic triples} \ \\

286 compact $4$-arithmetic triples:

\begin{multicols}{5}
(2, 3, 25)

(2, 3, 29)

(2, 3, 33)

(2, 3, 35)

(2, 3, 39)

(2, 3, 44)

(2, 3, 45)

(2, 3, 46)

(2, 3, 52)

(2, 3, 56)

(2, 3, 70)

(2, 3, 72)

(2, 3, 78)

(2, 3, 84)

(2, 3, 90)

(2, 4, 17)

(2, 4, 19)

(2, 4, 21)

(2, 4, 27)

(2, 4, 32)

(2, 4, 34)

(2, 4, 40)

(2, 4, 48)

(2, 4, 60)

(2, 4, 66)

(2, 5, 11)

(2, 5, 13)

(2, 5, 22)

(2, 5, 24)

(2, 5, 25)

(2, 5, 28)

(2, 5, 36)

(2, 5, 40)

(2, 5, 42)

(2, 5, 45)

(2, 5, 50)

(2, 5, 90)

(2, 6, 13)

(2, 6, 21)

(2, 6, 22)

(2, 6, 26)

(2, 6, 28)

(2, 6, 36)

(2, 6, 42)

(2, 7, 18)

(2, 7, 30)

(2, 7, 35)

(2, 7, 70)

(2, 7, 84)

(2, 8, 14)

(2, 8, 15)

(2, 8, 20)

(2, 8, 30)

(2, 8, 32)

(2, 8, 40)

(2, 8, 48)

(2, 10, 14)

(2, 10, 18)

(2, 10, 24)

(2, 10, 60)

(2, 11, 12)

(2, 11, 33)

(2, 11, 44)

(2, 11, 66)

(2, 12, 14)

(2, 12, 15)

(2, 12, 16)

(2, 12, 18)

(2, 12, 20)

(2, 12, 30)

(2, 12, 60)

(2, 13, 78)

(2, 14, 21)

(2, 14, 42)

(2, 15, 20)

(2, 15, 60)

(2, 16, 24)

(2, 16, 48)

(2, 17, 17)

(2, 18, 54)

(2, 19, 19)

(2, 20, 30)

(2, 20, 60)

(2, 21, 21)

(2, 22, 44)

(2, 23, 46)

(2, 26, 52)

(2, 27, 27)

(2, 28, 56)

(2, 32, 32)

(2, 34, 34)

(2, 35, 70)

(2, 36, 72)

(2, 39, 78)

(2, 40, 40)

(2, 42, 84)

(2, 45, 90)

(2, 48, 48)

(2, 60, 60)

(2, 66, 66)

(3, 3, 22)

(3, 3, 23)

(3, 3, 26)

(3, 3, 28)

(3, 3, 35)

(3, 3, 36)

(3, 3, 39)

(3, 3, 42)

(3, 3, 45)

(3, 4, 11)

(3, 4, 14)

(3, 4, 15)

(3, 4, 16)

(3, 4, 18)

(3, 4, 24)

(3, 4, 30)

(3, 4, 60)

(3, 5, 8)

(3, 5, 11)

(3, 5, 12)

(3, 5, 13)

(3, 5, 20)

(3, 5, 25)

(3, 5, 30)

(3, 6, 7)

(3, 6, 9)

(3, 6, 14)

(3, 6, 42)

(3, 7, 14)

(3, 7, 42)

(3, 8, 10)

(3, 8, 12)

(3, 8, 16)

(3, 8, 40)

(3, 10, 12)

(3, 10, 15)

(3, 11, 22)

(3, 11, 33)

(3, 12, 60)

(3, 13, 13)

(3, 13, 39)

(3, 15, 30)

(3, 15, 45)

(3, 21, 21)

(3, 22, 22)

(3, 24, 72)

(3, 26, 26)

(3, 26, 78)

(3, 28, 28)

(3, 28, 84)

(3, 30, 90)

(3, 36, 36)

(3, 42, 42)

(4, 4, 16)

(4, 4, 17)

(4, 4, 20)

(4, 4, 24)

(4, 4, 30)

(4, 4, 33)

(4, 5, 8)

(4, 5, 15)

(4, 5, 30)

(4, 5, 60)

(4, 6, 7)

(4, 6, 9)

(4, 6, 14)

(4, 6, 18)

(4, 6, 60)

(4, 7, 12)

(4, 7, 14)

(4, 7, 28)

(4, 8, 40)

(4, 9, 12)

(4, 9, 18)

(4, 9, 36)

(4, 10, 15)

(4, 11, 44)

(4, 12, 30)

(4, 14, 14)

(4, 15, 15)

(4, 15, 20)

(4, 20, 20)

(4, 30, 30)

(4, 30, 60)

(4, 32, 32)

(4, 40, 40)

(4, 48, 48)

(5, 5, 11)

(5, 5, 12)

(5, 5, 14)

(5, 5, 18)

(5, 5, 20)

(5, 5, 21)

(5, 5, 25)

(5, 5, 45)

(5, 6, 15)

(5, 8, 10)

(5, 8, 24)

(5, 10, 15)

(5, 12, 20)

(5, 12, 60)

(5, 14, 14)

(5, 15, 30)

(5, 18, 18)

(5, 24, 24)

(5, 60, 60)

(6, 6, 11)

(6, 6, 13)

(6, 6, 14)

(6, 6, 18)

(6, 6, 21)

(6, 7, 14)

(6, 10, 15)

(6, 11, 11)

(6, 12, 60)

(6, 14, 14)

(6, 14, 21)

(6, 15, 15)

(6, 16, 16)

(6, 18, 18)

(6, 20, 20)

(6, 21, 42)

(6, 30, 30)

(6, 60, 60)

(7, 7, 9)

(7, 7, 15)

(7, 7, 35)

(7, 7, 42)

(7, 8, 8)

(7, 10, 10)

(7, 12, 12)

(7, 21, 21)

(7, 42, 42)

(8, 8, 10)

(8, 8, 15)

(8, 8, 16)

(8, 8, 20)

(8, 8, 24)

(8, 12, 12)

(8, 16, 48)

(8, 24, 24)

(8, 48, 48)

(9, 10, 10)

(9, 12, 12)

(9, 12, 36)

(9, 54, 54)

(10, 10, 12)

(10, 10, 30)

(10, 12, 12)

(10, 15, 15)

(10, 20, 60)

(10, 30, 30)

(10, 60, 60)

(11, 11, 22)

(11, 11, 33)

(11, 22, 66)

(11, 44, 44)

(12, 12, 15)

(12, 12, 30)

(12, 16, 16)

(13, 13, 39)

(13, 52, 52)

(14, 14, 21)

(14, 56, 56)

(15, 15, 30)

(15, 20, 20)

(16, 16, 24)

(16, 32, 32)

(17, 34, 34)

(18, 18, 27)

(18, 72, 72)

(20, 20, 30)

(20, 40, 40)

(21, 84, 84)

(22, 22, 22)

(23, 23, 23)

(24, 48, 48)

(26, 26, 26)

(28, 28, 28)

(30, 60, 60)

(33, 66, 66)

(35, 35, 35)

(36, 36, 36)

(39, 39, 39)

(42, 42, 42)

(45, 45, 45)
\end{multicols}

31 noncompact $4$-arithmetic triples:

\begin{multicols}{5}
(2, 15, oo)

(2, 16, oo)

(2, 20, oo)

(2, 24, oo)

(2, 30, oo)

(3, 8, oo)

(3, 10, oo)

(3, 12, oo)

(3, 15, oo)

(4, 5, oo)

(4, 8, oo)

(4, 10, oo)

(4, 20, oo)

(5, 6, oo)

(5, 10, oo)

(5, 15, oo)

(6, 8, oo)

(6, 10, oo)

(6, 12, oo)

(6, 30, oo)

(8, oo, oo)

(8, 24, oo)

(10, oo, oo)

(10, 30, oo)

(12, oo, oo)

(15, oo, oo)

(15, 15, oo)

(16, 16, oo)

(20, 20, oo)

(24, 24, oo)

(30, 30, oo)
\end{multicols}

\subsubsection{$5$-arithmetic triples}\ \\

94 compact $5$-arithmetic triples:

\begin{multicols}{5}
(2, 3, 31)

(2, 3, 58)

(2, 3, 62)

(2, 3, 64)

(2, 3, 96)

(2, 3, 102)

(2, 4, 23)

(2, 4, 25)

(2, 4, 38)

(2, 4, 44)

(2, 4, 50)

(2, 4, 54)

(2, 5, 26)

(2, 5, 70)

(2, 6, 17)

(2, 6, 32)

(2, 6, 48)

(2, 7, 11)

(2, 8, 11)

(2, 8, 13)

(2, 8, 22)

(2, 9, 14)

(2, 9, 16)

(2, 9, 27)

(2, 9, 54)

(2, 9, 72)

(2, 9, 90)

(2, 10, 16)

(2, 10, 40)

(2, 23, 23)

(2, 25, 25)

(2, 29, 58)

(2, 31, 62)

(2, 32, 64)

(2, 38, 38)

(2, 44, 44)

(2, 48, 96)

(2, 50, 50)

(2, 51, 102)

(2, 54, 54)

(3, 3, 29)

(3, 3, 31)

(3, 3, 32)

(3, 3, 48)

(3, 3, 51)

(3, 6, 16)

(3, 6, 24)

(3, 12, 24)

(3, 17, 17)

(3, 32, 32)

(3, 32, 96)

(3, 34, 102)

(3, 48, 48)

(4, 4, 19)

(4, 4, 22)

(4, 4, 25)

(4, 4, 27)

(4, 11, 11)

(4, 13, 13)

(4, 18, 36)

(4, 22, 22)

(5, 5, 13)

(5, 5, 35)

(5, 6, 18)

(5, 10, 20)

(5, 16, 16)

(5, 18, 30)

(5, 40, 40)

(6, 6, 16)

(6, 6, 24)

(6, 7, 42)

(6, 12, 36)

(6, 16, 48)

(7, 9, 9)

(8, 8, 11)

(8, 9, 9)

(8, 10, 10)

(9, 9, 27)

(9, 9, 36)

(9, 9, 45)

(10, 10, 20)

(12, 16, 48)

(12, 18, 36)

(16, 64, 64)

(19, 38, 38)

(22, 44, 44)

(24, 96, 96)

(25, 50, 50)

(27, 54, 54)

(29, 29, 29)

(31, 31, 31)

(32, 32, 32)

(48, 48, 48)

(51, 51, 51)
\end{multicols}

6 noncompact $5$-arithmetic triples:

\begin{multicols}{5}
(2, 11, $\infty$)

(2, 22, $\infty$)

(3, 11, $\infty$)

(11, $\infty$, $\infty$)

(11, 11, $\infty$)

(22, 22, $\infty$)

\end{multicols}

\subsection{Number of $r$-arithmetic triples for $r\leq 15$}

\begin{center}
  \begin{tabular}{ l || c | c }
    $r$ & $\#\calT(r)_{0}$ & $\#\calT(r)_{\infty}$ \\ \hline \hline
    1 & 76 & 9\\ \hline
    2 & 148 & 16\\ \hline
    3 & 111 & 13\\ \hline
    4 & 286 & 31\\ \hline
    5 & 94 & 6\\ \hline
    6 & 430 & 37\\ \hline
    7 & 100 & 0 \\ \hline
    8 & 435 & 48\\ \hline
    9 & 89 & 16\\ \hline
    10 & 558 & 28\\ \hline
    11 & 83 & 6\\ \hline
    12 & 699 & 92\\ \hline
    13 & 87 & 0\\ \hline
    14 & 666 & 6\\ \hline
    15 & 86 & 8
  \end{tabular}
\end{center}

The following graphs plot $\#\calT(r)_{0}$ against $r$ for odd $r$ and even $r$, respectively.
\begin{center}
\includegraphics[scale=.4]{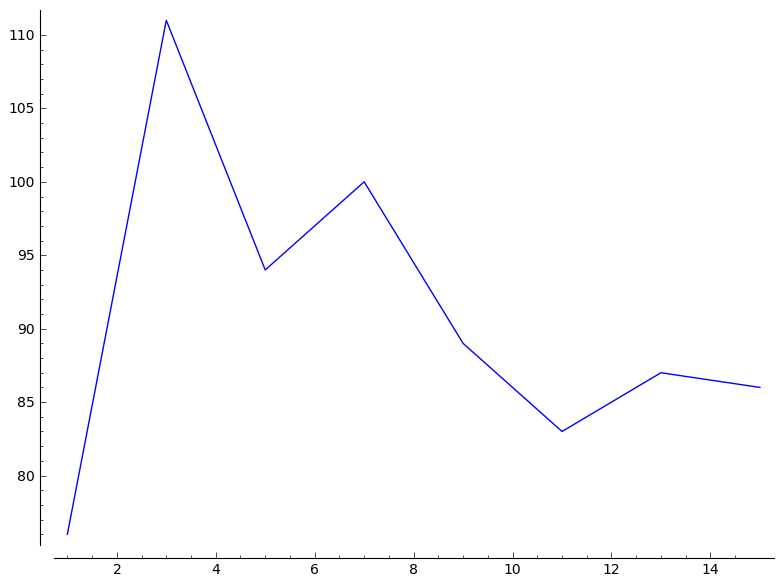}
\includegraphics[scale=.4]{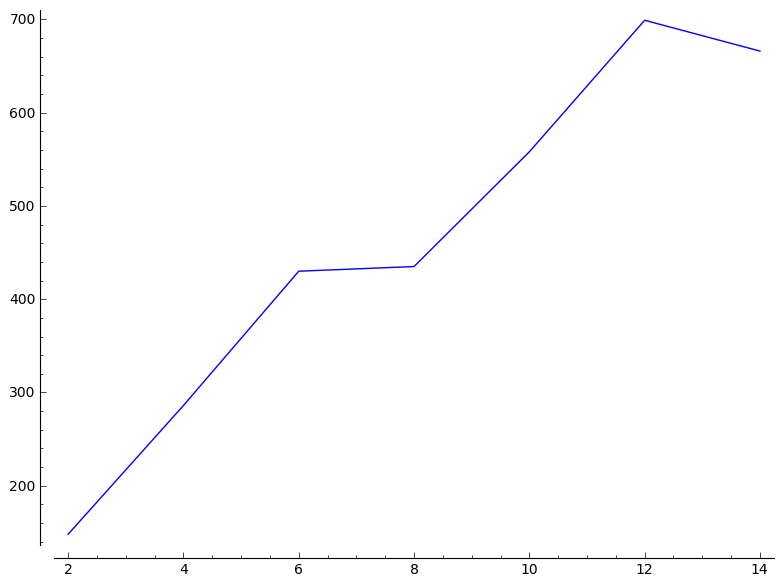}
\end{center}

It appears that the function $r \mapsto \#\calT(r)$ depends on the factorization of $r$ into primes, but we have not been able to provide a convincing growth rate.  Nevertheless, there is evidence for the following conjecture.

\begin{conj*}
The set $\calT(r)$ is nonempty for all $r \geq 1$.
\end{conj*}

By Corollary \ref{cor:adimoooo}, any integer $r$ of the form $r=\phi(a)/\!\gcd(2,a)$ (a kind of totient number) yields $(a,\infty,\infty) \in \calT(r)$.

\end{document}